\documentclass[11pt]{article}
\usepackage[margin=1in,a4paper]{geometry}
\usepackage[utf8]{inputenc}

\usepackage{amssymb,amsmath}
\usepackage{amsthm} 
\usepackage{tikz-cd}

\usepackage{xcolor}
\usepackage{esint}
\usepackage[colorlinks,linkcolor=blue, urlcolor  = blue,citecolor=blue,pagebackref,hypertexnames=false, breaklinks]{hyperref}

\usepackage{mathtools}
\usepackage{float}
\mathtoolsset{showonlyrefs}

\usepackage[shortlabels]{enumitem}
\setlist[enumerate,1]{wide, labelindent=0pt,label={\upshape(\roman*)}}
\usepackage{comment}

\usepackage[nottoc,notlot,notlof]{tocbibind}

\usepackage{comment}

\newcommand{\ang}[1]{\left<#1\right>} 

    \def\XXint#1#2#3{{\setbox0=\hbox{$#1{#2#3}{\int}$}
    \vcenter{\hbox{$#2#3$}}\kern-.5\wd0}}

\usepackage{accents}

\newcommand{\vardbtilde}[1]{\widetilde{\raisebox{0pt}[0.85\height]{$\widetilde{#1}$}}}

\theoremstyle{plain}
\newtheorem{thm}{Theorem}[section]
\newtheorem{lem}[thm]{Lemma}
\newtheorem{cor}[thm]{Corollary}
\newtheorem{prop}[thm]{Proposition}

\newtheorem{example}[thm]{Example}
\newtheorem{assump}[thm]{Assumption}

\DeclareMathOperator{\supp}{supp}

\newenvironment{Example}{\begin{example}\rm}{\end{example}}

\theoremstyle{definition}
\newtheorem{defn}[thm]{Definition}

\theoremstyle{remark}
\newtheorem{remark}[thm]{Remark}
\newcommand{\bremark}{\begin{remark} \em}
\newcommand{\eremark}{\end{remark} }
\usepackage{color}

\hbadness=\maxdimen

\title{Riesz transform, function spaces and their applications on infinite dimensional compact groups}
\author{Alexander Bendikov, Li Chen\footnote{Research partially supported by grant 10.46540/4283-00175B from Independent Research Fund Denmark} ~and Laurent Saloff-Coste\footnote{Research partially supported by National Science Foundation grants DMS-2054593 and DMS-2343868}}
\date{}

\begin{document}

\maketitle
\vspace{-1em}
\begin{abstract}
On a compact connected group $G$, consider the infinitesimal generator $-L$ of a central symmetric Gaussian convolution semigroup $(\mu_t)_{t>0}$. We establish several regularity results of the solution to the Poisson equation $LU=F$, both in strong and weak senses. To this end, we introduce  two classes of Lipschitz spaces for $1\le p\le \infty$: $\Lambda_{\theta}^p$, defined via the associated Markov semigroup, and $\mathrm L_{\theta}^p$, defined via the intrinsic distance. In the strong sense, we prove a priori Sobolev regularity and Lipschitz regularity in the class of $\Lambda_{\theta}^p$ space.  In the distributional sense, we further show local regularity in the class of $\mathrm L_{\theta}^{\infty}$ space. These results require some strong assumptions on $-L$. Our main techniques build on the differentiability of the associated semigroup, explicit dimension-free $L^p$ ($1<p<\infty$) boundedness of first and second order Riesz transforms, and a comparison between the two Lipschitz norms.    
\end{abstract}
\tableofcontents
\section{Introduction}
It is well known that if $f$ is merely a continuous function, in a bounded domain $\Omega\subset R^n$ with $n>1$, it is possible that the equation $\Delta u=f$  has no classical solution (i.e., there are no $\mathcal C^2$ function $u$ solving the equation).   However, if instead of continuity we require $f$ to be H\"older continuous with exponent $\alpha\in (0,1)$ then the equation $\Delta u=f$, interpreted in the sense of distributions, admits a unique solution $u$ which is actually $\mathcal C^{2,\alpha}$.

This is a basic steps in the classical Schauder interior regularity theory for second order elliptic operators and more sophisticated versions apply  when $\Delta$ is the Laplace operator on a Riemannian manifold.

This paper is concerned with this problem but in a more exotic setting where the underlying space is a compact connected locally connected metrizable group such as the infinite dimensional torus $\mathbb T^\infty$, i.e., the product of countably many circles, and other products such as $\prod_{n\ge 5} \mbox{SO}(n)$.  On such groups, there are many natural Laplacians, each with possibly drastically different properties,  and one can study properties of the solutions of the equation $\Delta u=f$, interpreted  in various ways, for such Laplacians. This is motivated by \cite[Chapter 5]{Bendikov1995}.

To avoid most technical details, let us focus here on $\mathbb T^\infty=\prod_{i=1}^\infty \mathbb T_i$, equipped with its normalized Haar measure. Laplacians are operators defined on smooth functions depending only of finitely many coordinates by 
$$\Delta_A\phi=-\sum_{i,j=1}^\infty a_{i,j} \partial_i\partial_j \phi$$
where $A=(a_{i,j})$ is an infinite symmetric positive definite matrix ($\sum_{i,j} a_{i,j}\xi_i\xi_j\ge 0$ and equal to $0$ only if $\xi=(\xi_i)=0$ for all vectors $\xi$ having finitely many non-zero coordinates).  The simplest of these Laplacian is $\Delta_0 \phi=-\sum_{i=1}^\infty \partial_i^2\phi$ but it behaves very badly from the point of view of the question raised above. On the other hand, $\Delta_1 \phi=-\sum 2^i\partial _i^2\phi$ is expected to behave much better, in some sense. In particular, while there does not exit a natural good metric on $\mathbb T^\infty$ associated with $\Delta_0$, there is a very nice metric, $d_1$, associated with $\Delta_1$ (its intrinsic metric, form the viewpoint of Dirichlet forms). Using this metric, we can ask whether or not  weak local solutions of the equation $\Delta_1 u=f$ in an open set $\Omega$ with $f$ H\"older continuous with respect to the metric $d_1$ in $\Omega$, are automatically classical solutions for which not only $\Delta_1 u$ makes sense as a continuous function in $\Omega$, but $\Delta_1u$ is given by the converging series $-\sum_{i=1}^\infty 2^i\partial _i^2u$
where each $\partial_i^2u$ exists and is continuous. In this paper, we prove that this is indeed the case for $\Delta_1$ (and many other Laplacians). This statement is definitely not true for all possible Laplacians on $\mathbb T^\infty$. As explained later below, this applies not only when the equation  $\Delta_1 u=f$ is interpreted in the standard weak sense in $\Omega$ for which one assumes $u$ is in $L^2$ with first derivatives in $L^2$, but even when $u$ and $f$ are only assumed to be distributions of a certain type (a class of distributions associated to $\Delta_1$ introduced and studied in \cite{BSC2005,BSC2006} and which, for example, always include Radon measures).

In order to prove such results, we rely on several interesting facts and ideas, some of which have to be developed from scratch in this setting. One key idea is taken from \cite{BSC2025} where the $L^1$-differentiability of the Markov semigroups $e^{-tL_A}$ associated with Laplacians is discussed. A second key ideas is to use the boundedness of associated Riesz transforms on $L^p$ with $p$ very close to $1$ or infinity (ideally, one would want to use the Riesz transforms on $L^1$ or $L^\infty$ but, of course, this is not possible). The important role played by the Riesz transforms in our arguments implies that we have to restrict ourselves to bi-invariant Laplacians (all Laplacian on $\mathbb T^\infty$ are bi-invariant but this is of course not the case when $G$ is not abelian). The third key idea (well-understood in the Euclidean case) is to relate two notions of ``H\"older spaces'' associated to a Laplacian $L_A$. One type of H\"older spaces is defined in terms of the action of the semigroup $e^{-tL_A}$ while the other is defined more classically using the intrinsic distance associated with $L_A$ (assuming this distance is well-defined). The main finding of this paper is that there is a large class of Laplacians  (on $\mathbb T^\infty$ but also on any other metrizable, locally connected connected compact group) for which these different steps can be implemented to obtain the type of result described above concerning the  continuity of local distributional solutions of $\Delta_A u=f$ when $f$ is H\"older continuous in $\Omega$.

\section{The setting}
Following \cite{BSC2001}, we introduce the setting of symmetric central Gaussian semigroups on compact connected groups. Let $G$ be a compact connected metrizable group equipped with its normalized Haar measure $\nu$. Denote by $e$ the neutral element in $G$. Then $G$ contains a descending family of compact normal subgroups $(K_{\alpha})_{\alpha\in \aleph}$ ($\aleph$ is either finite or countable) such that $\bigcap_{\alpha\in \aleph}K_{\alpha}=\{e\}$. For each $\alpha\in \aleph$, $G/K_{\alpha}=G_{\alpha}$ is a Lie group. Consider the projection maps $\pi_{\alpha,\beta}: G_{\beta} \to G_{\alpha}, \beta\ge \alpha$ and $\pi_{\alpha}:G \to G_{\alpha}$. The group $G$ is then the projective limit of the projective system $(G_{\alpha}, \pi_{\alpha,\beta})_{\beta\ge \alpha}$.
The Lie algebra $\mathfrak G$ is defined as the projective limit of the Lie algebras $\mathfrak G_{\alpha}$ of the groups $G_{\alpha}$ equipped with the projective maps $d\pi_{\alpha, \beta}$.

For a compact Lie group $N$, denote by $\mathcal C_0^{\infty}(N)$ the set of all smooth functions on $N$. The space of Bruhat test functions $\mathcal B(G)$, introduced in \cite{Bruhat}, is defined by 
\[
\mathcal B(G)=\{f: G\to \mathbb R;  f=\phi\circ \pi_{\alpha} \text{ for some } \alpha\in \aleph  \text{ and } \phi \in \mathcal C_0^{\infty}(G_{\alpha})\}.
\]

Let $(\mu_t)_{t> 0}$ be a weakly continuous Gaussian convolution semigroup of probability measures on $G$. That is, each  $\mu_t$, $t>0$, is a probability measure on $G$ and $(\mu_t)_{t>0}$ satisfies
\begin{itemize}
    \item $\mu_t*\mu_s=\mu_{t+s}$, $\forall t,s>0$;
    \item $\mu_t\to \delta_e$ weakly as $t\to 0$;
    \item $t^{-1}\mu_t(V^c)\to 0$ as $t\to 0$ for any neighborhood $V\ni e$.
\end{itemize}
\begin{assump}\label{assump}
Let $(\mu_t)_{t> 0}$ be a weakly continuous Gaussian convolution semigroup of probability measures on $G$. Throughout the paper, we assume the following:
\begin{itemize}
    \item $(\mu_t)_{t> 0}$ is {\bf symmetric}, i.e., $\mu_t(A)=\mu_t(A^{-1})$ for all $t>0$ and all Borel subset $A\subset G$;
    \item $(\mu_t)_{t> 0}$ is {\bf central}, i.e., $\mu_t(a^{-1}Aa)=\mu_t(A)$ for all $t>0$, all $a\in G$ and all Borel subset $A\subset G$;
    \item $(\mu_t)_{t> 0}$ is {\bf non-degenerate}, i.e., the support of $\mu_t$ is $G$ for all $t>0$.
\end{itemize}
\end{assump}

Associated with $(\mu_t)_{t> 0}$ there exists a Brownian motion on $G$, denoted by $(B_t)_{t\ge0}$. That is, $(B_t)_{t\ge0}$ is a $G$-valued process having independent stationary increments, continuous paths, and which is non-degenerate and invariant under the mappings $x\mapsto x^{-1}$ and $x\mapsto axa^{-1} $, $a\in G$. 



\begin{defn}
Let $(\mu_t)_{t> 0}$ be a Gaussian convolution semigroup on $G$. For each $\alpha\in \aleph$, we define $(\mu_{t}^\alpha)_{t> 0}$ to be the Gaussian  convolution semigroup on $G_{\alpha}$ by setting 
\[
\mu_{t}^\alpha(V)=\mu_t(\pi_{\alpha}^{-1}(V)), \quad \text{ for any open set } V\subset G_{\alpha}.
\]
\end{defn} 
Conversely, one can construct $(\mu_t)_{t> 0}$ on $G$ as the unique projective limit of a projective system of convolution semigroups $(\mu_{t}^{\alpha})_{t > 0}$. We refer to \cite[Properties 6.2.4]{Heyer} for more details. 
Next we introduce the notion of projective basis, see \cite{Born1}. 
\begin{defn}
     A family of left-invariant vector fields $(Y_i)_{i\in I}$ on $\mathfrak G$ is a projective basis of $\mathfrak G$ with respect to $(K_{\alpha})_{\alpha\in \aleph}$ if for any $\alpha\in \aleph$, there is a finite subset $I_{\alpha}\subset  I$ such that $d\pi_{\alpha}(Y_i)=Y_{\alpha,i}$, $j\in I_{\alpha}$, form a basis of $\mathfrak G_{\alpha}$ and $d\pi_{\alpha}(Y_i)=0$ for $i\notin I_{\alpha}$.
\end{defn}
Since $K_{\alpha}$, $\alpha\in \aleph$, is descending, there exists a  projective basis $(Y_i)_{i\in I}$ on $\mathfrak G$ with respect to $(K_{\alpha})_{\alpha\in \aleph}$ with  $I=\bigcup_{\alpha\in \aleph}I_{\alpha}$.

From now on we consider a weakly continuous Gaussian convolution semigroup $(\mu_t)_{t> 0}$ on $G$ and assume that $(\mu_t)_{t>0}$ is symmetric, central and Gaussian. Set 
\begin{equation}\label{eq:heat semigroup}
H_tf(x)=\int_G f(xy) d\mu_t(y).   
\end{equation}
Then $(H_t)_{t>0}$ forms a Markov semigroup and extends to $L^2(G, d\nu)$ as a self-adjoint semigroup. Associated with $\mu_t$ there is an $L^2(G, d\nu)$-infinitesimal generator $(-L, \mathrm{Dom}(L))$ and a Dirichlet form $(\mathcal E, \mathrm{Dom}(\mathcal E))$ such that $H_t=e^{-tL}$ on $L^2(G, d\nu)$ and 
\[
\mathcal E(f,g)=\langle L^{1/2}f, L^{1/2}g\rangle, \quad f,g \in  \mathrm{Dom}(\mathcal E)= \mathrm{Dom}(L^{1/2}).
\]
The assumption that $(\mu_t)_{t>0}$ is central is equivalent to the fact that $L$ is {\em bi-invariant}.

Let  $(Y_i)_{i\in I}$ be a projective basis of the Lie algebra $\mathfrak G$. Then $L$ is a second order left-invariant differential operator of the form
\[
L=-\sum_{i,j\in  I} a_{i,j} Y_iY_j
\]
where $A=(a_{i,j})_{I\times I}$ is a real symmetric positive-definite matrix in the sense that $a_{i,j}=a_{j,i}\in \mathbb R$ and $\sum a_{i,j}\xi_i\xi_j> 0$ for any non-zero $\xi \in \mathbb R^{(I )}$. We note that there exists an upper triangular matrix $T=(t_{i,j})_{I\times I}$ such that $A=T^tT$ (see \cite[Lemma 2.3]{BSC2002}). 

\begin{thm}[\protect{\cite[Theorem 2.4]{BSC2002}}]
Fix a projective basis $(Y_i)_{i\in I}$. Let $L=-\sum_{i,j\in I} a_{i,j} Y_iY_j$ with $A$ symmetric and positive-definite. Then the family $(X_i)_{i\in I}$  given by 
\[
X_i=\sum_{j\in I} t_{i,j}Y_j, \quad i\in I
\]
is a projective basis of $\mathfrak G$ and yields a square field decomposition  $-L=\sum_{i\in I}X_i^2$.
\end{thm}

Fix the choice of $(X_i)_{i\in I}$ as above. We note that, because we assume that $L$ is bi-invariant,  $(X_i)_{i\in I}$ commutes with the operator $L$. For any given $f=\phi\circ \pi_{\alpha}\in \mathcal B(G)$, the operator $L$ reads
\[
Lf=-\sum_{i\in I_{\alpha}} X_{\alpha,i}^2 \phi \circ \pi_{\alpha}=(L_{\alpha}\phi)\circ \pi_{\alpha},
\]
where $X_{\alpha,i}=d\pi_{\alpha}(X_i)$, $I_{\alpha}=\{i\in I: d\pi_{\alpha}(X_i)\ne 0\}$, and $-L_{\alpha}$ is the infinitesimal generator of $(\mu_{t}^\alpha)_{t>0}$ on $G_{\alpha}$. Let $(H_{t}^\alpha)_{t>0}$ be the associated Markov semigroup on $G_{\alpha}$ defined by 
\[
H_{t}^\alpha\phi(x)=\int_{G_{\alpha}} \phi(xy)d\mu_{t}^\alpha(y).
\]
Then $(H_t)_{t>0}$ is also a projective limit of the sequence $(H_{t}^\alpha)_{t>0}$, $\alpha\in \aleph$ and for any $f=\phi\circ \pi_{\alpha}\in \mathcal B(G)$, one has $H_t f=(H_{t}^\alpha\phi)\circ \pi_{\alpha}$.

The Dirichlet form $(\mathcal E, \mathrm{Dom}(\mathcal E))$ is regular and strictly local and $\mathcal B(G)$ is a core in $\mathrm{Dom}(\mathcal E)$.
Define the field operator $\Gamma$ to be the symmetric bilinear form $\Gamma(f,g)=\frac12(-L(fg)+fLg+gLf)$ on the space $\mathcal B(G)$. Then 
\[
\Gamma(f,g)=\sum_{i\in I} (X_i f)(X_i g)
\]
and 
\[
\mathcal E (f,g)=\int_{G}\Gamma(f,g) d\nu=\int_{G}\sum_{i\in I} (X_i f)(X_i g)d\nu.
\]
Similarly, we denote by $(\mathcal E_{\alpha}, \mathrm{Dom}(\mathcal E_{\alpha}))$ the Dirichlet form associated to $\mu_t^{\alpha}$ and 
\[
\Gamma_{\alpha}(\phi,\psi)=\sum_{i\in I} (X_{\alpha,i} \phi)(X_{\alpha,i}\psi), \quad\forall \phi,\psi\in \mathcal C_0^{\infty}(G_{\alpha}).
\]
Then one has $\mathcal E_{\alpha}(\phi,\psi)=\int_{G_{\alpha}}\Gamma_{\alpha}(\phi,\psi)d\nu_{\alpha}$.

Finally we fix a Brownian motion $B=((B_t)_{t\ge 0}, \mathbf P)$ on $G$ with law $\mu_t$. One has
\[
H_tf(x)=\mathbf E(f(xB_t))=\mathbf E(f(B_t)|B_0=x)=\int_G f(xy) d\mu_t(y).
\]
Moreover, there exists a sequence of Brownian motions $(B_t^{\alpha})_{t\ge 0}$, $\alpha\in \aleph$,  on $G_\alpha$ with the law $\mu_t^{\alpha}$ such that for $x\in G_{\alpha}$,
\[
H_t^{\alpha}f(x)=\mathbf E(f(xB_t^{\alpha}))=\mathbf E(f(B_t^{\alpha})|B_0^{\alpha}=x)=\int_{G_\alpha} f(xy) d\mu_t^{\alpha}(y).
\]

\paragraph{Infinite-dimensional torus.}
Consider the infinite-dimensional torus $\mathbb T^{\infty}$, that is, the product of countably many copies of $\mathbb T=\mathbb R\setminus 2\pi\mathbb Z$. The topology on $\mathbb T^{\infty}$ is the product topology generated by cylindric sets. We will regard $\mathbb T^{\infty}$ as a compact connected, locally connected Abelian group equipped with normalized Haar measure $d\nu=dx$. Also, $\mathbb T^{\infty}$ is the projective limit of finite dimensional tori.  In fact, for any $k\in \mathbb N$, there is a canonical map $\pi_k$ from $\mathbb T^{\infty}$ to $\mathbb T^k$ such that $\pi_k(\mathbb T^{\infty})=\mathbb T^{k}$.

Consider  a symmetric infinite positive-definite matrix $A=(a_{i,j})$.
Associated with $A$ there exists a closable quadratic form
\[
\mathcal E_A(u,u)=\int_{\mathbb T^{\infty}} \sum_{i,j}a_{i,j} \partial_i u\partial_j u d\mu, \quad u\in\mathcal B(\mathbb T^{\infty})
\]
whose minimal closure is a strictly local invariant Dirichlet form (see \cite[Theorem 5.1]{BSC2000}). Here $\mathcal C$ is a core. Moreover, there exists an ``upper triangular'' sequence $\tau=(\tau_i)$ of vectors $\tau_i=(\tau_{i,k})_{k=1}^{\infty}\in \mathbb R^{\infty}$ such that $(A\xi,\xi)=\sum_{i=1}^{\infty}(\tau_i, \xi)^2$ and 
\[
\mathcal E_A(u,u)=\int_{\mathbb T^{\infty}} \sum_k |X_k u|^2 d\mu, \quad u\in\mathcal B(\mathbb T^{\infty}),
\]
where $X_iu=\partial_{\tau_i} u=(\tau_i, \partial u)$. In other words, $\mathcal E$ admits a carr\'e du champ $\Gamma: \mathrm{Dom}(\mathcal E)\times \mathrm{Dom}(\mathcal E)\to L^1(\mu)$ such that $\Gamma(u,u)=\sum_k |\partial_{\tau_k} u|^2$.
Denote by $-L_A$ the $L^2$-generator associated to $A$ (or $\mathcal E_A$), then 
\[
-L_Au=\sum_{i,j}a_{i,j} \partial_i\partial_j u=\sum_k \partial_{\tau_k}^2 u,  \quad \forall u\in\mathcal B(\mathbb T^{\infty}).
\]

From now on, if not otherwise stated, we always assume that $(\mu_t)_{t>0}$ is a Gaussian convolution semigroup on the compact metrizable group $G$ satisfying the standing Assumption \ref{assump}. In particular, $(\mu_t)_{t>0}$ is central is equivalent to the fact that $L$ is {\em bi-invariant}

\section{\texorpdfstring{$L^1$}{L1} differentiability of the heat semigroup}
\subsection{Heat kernel estimates}
We consider several properties which may be satisfied by the Gaussian semigroup $(\mu_t)_{t>0}$. These properties was introduced in \cite[Section 2]{BSC2000} for more general settings where $(\mu_t)_{t>0}$ is not necessarily central (or $L$ may not be bi-invariant). 
\begin{defn}\label{def:CK etc}
Consider a symmetric central Gaussian semigroup $(\mu_t)_{t>0}$ on $G$.
\begin{enumerate}
    \item We say that $(\mu_t)_{t>0}$ has property $(\mathrm{AC})$ if for all $t>0$, $\mu_t$ is absolutely continuous with respect to the Haar measure $\nu$.
    \item We say that $(\mu_t)_{t>0}$ has property $(\mathrm{CK})$ if it satisfies $(\mathrm{AC})$ and has a continuous density $x\to \mu_t(x)$. 
    
    \item We say that $(\mu_t)_{t>0}$ satisfies the property $(\mathrm{CK*})$ if it has property $(\mathrm{CK})$ and the density $\mu_t(x)$ satisfies
    \[
    \lim_{t\to 0}t \log \mu_t(e)=0.
    \]

    \item We say that $(\mu_t)_{t>0}$ satisfies the property $(\mathrm{CK\lambda})$ if property $(\mathrm{CK})$ holds and the density $\mu_t(x)$ satisfies
     \[
    \sup_{0<t<1}\{t^{\lambda} \log \mu_t(e)\}< \infty.
    \]

    \item We say that $(\mu_t)_{t>0}$ satisfies the property $(\mathrm{CK0^{+}})$ if property $(\mathrm{CK\lambda})$ holds for all $\lambda>0$.
\end{enumerate}    
\end{defn}
\begin{remark}\label{rem:spectral gap}
We note that the property $(\mathrm{CK})$ implies the existence of positive $L^2$-spectral gap. For more details, we refer to  \cite[Theorem 5.3]{BSC2003FM}.
\end{remark}

\begin{Example}
On the infinite-dimensional torus $\mathbb T^{\infty}$, consider the heat semigroup $H_t$ associated with infinitesimal generator $L=-\sum a_i\partial_i^2$. One can summarize from \cite[Section 4.2]{BSC2025} that
\begin{enumerate}[(a)]
    \item The property $(\mathrm{CK*})$ holds if and only if $\log\#\{i:a_i\le s\}=o(s)$ as $s$ tends to infinity. Hence we must have $\lim_{i\to \infty} a_i=\infty$.
    \item If $a_i=i^{1/\lambda}$ for some $\lambda>0$, then property $(\mathrm{CK\lambda})$ is satisfied. 
    \item If $a_i=2^{i^{\sigma}}$ for some $\sigma>0$, then property $(\mathrm{CK0^+})$ holds.  
\end{enumerate}
For a description of these properties for symmetric central Gaussian semigroups on a general metrizable, locally connected, connected compact group, see \cite{BSC2001}. On any such group, there are many bi-invariant Laplacians whose associate symmetric central Gaussian convolution semigroup $(\mu_t)_{t>0}$ satisfies $(\mathrm{CK0^+})$, for example.\end{Example}

\begin{defn}\label{def:distance}
Given a symmetric Gaussian semigroup $(\mu_t)_{t>0}$ on $G$ with infinitesimal generator $-L$. We set
\[
d(x,y)=\sup\{|f(x)-f(y)|:f\in \mathcal B(G), \Gamma(f,f)\le 1\}
\]
and 
\[
\delta(x,y)=\sup\{|f(x)-f(y)|:f\in \mathcal B(G), \Gamma(f,f)\le 1, |Lf|\le 1\}.
\]
The quasi distances $d$ and $\delta$ are called respectively the intrinsic distance and the relaxed distance associated with $-L$.
\end{defn}

Since $L$ is bi-invariant, the quasi distances  $d$ and $\delta$ are also bi-invariant. Let  $d_{\alpha}$ be the intrinsic distance on $G_{\alpha}$ associated to $(\mathcal E_{\alpha}, \mathrm{Dom}(\mathcal E_{\alpha}))$. Then $d_{\alpha}$ are also bi-invariant distance and
\begin{equation}\label{eq:dist-sup}
    d(e,x)=\sup_{\alpha\in \aleph} d_{\alpha}(e_{\alpha},\pi_{\alpha}(x)), \quad \forall x\in  G,
\end{equation}
see \cite[Theorem 4.2]{BSC-AFST}.

\begin{thm}[\cite{BSC2000}]\label{thm:HKbounds}
Let $(\mu_t)_{t>0}$ be a symmetric Gaussian semigroup satisfying property $(\mathrm{CK})$.
\begin{enumerate}
    \item Assume that $(\mu_t)_{t>0}$ satisfies $(\mathrm{CK*})$, then the relaxed distance $\delta$ is a continuous distance function which defines the topology of $G$ and
    \begin{equation}\label{eq:HKboundCK*}
    \mu_t(x)\le \exp\left(M(t)-\frac{\delta(e,x)^2}{Ct}\right), \quad \forall t\in (0,1), \forall x\in G,
    \end{equation}
    where $M$ satisfies $\lim\limits_{t\to 0} tM(t)=0$. 
   
    \item Fix $\lambda\in (0,1)$. Assume that $(\mu_t)_{t>0}$ satisfies $(\mathrm{CK}\lambda)$, then the intrinsic distance $d$ is a continuous distance function which defines the topology of $G$ and
    \begin{equation}\label{eq:HKbound}
    \mu_t(x)\le \exp\left(\frac{A}{t^{\lambda}}-\frac{d(e,x)^2}{Ct}\right),\quad \forall t\in (0,1), \forall x\in G.
    \end{equation}
\end{enumerate}
\end{thm}

Next we discuss the heat kernel derivative estimates. Recall the iterated gradient $\Gamma_n$ ($n=1,2,3,\cdots$) recursively defined by 
\[
\Gamma_n(f,g)=\frac12(-L\Gamma_{n-1}(f,g)+\Gamma_{n-1}(f,Lg)+\Gamma_{n-1}(Lf,g)),
\]
with $\Gamma_0(f,g)=fg$. Since $L$ is bi-invariant,the higher iterated gradients given are given by
\[
\Gamma_n(f,g)=\sum_{(\ell_1, \cdots,\ell_n)\in I^n} (X_{\ell_1}\cdots X_{\ell_n} f)(X_{\ell_1}\cdots X_{\ell_n} g), \quad f,g \in \mathcal B(G).
\]
For convenience, we use the notation
\[
D_x^nf(X_{\ell_1}\cdots X_{\ell_n}) f=X_{\ell_1}\cdots X_{\ell_n} f(x)
\]
and set 
\begin{equation}\label{eq:Dn}
|D^n f|_L^2=\Gamma_n(f,f).
\end{equation}
\begin{cor}[{\cite[Theorem 4.2]{BSC2002}}] \label{cor:knHKbounds}
Under the hypothesis of Theorem \ref{thm:HKbounds}, for each fixed $k$ and $n$, there exists $C=C(k,n)$ such that the estimates \eqref{eq:HKboundCK*} and \eqref{eq:HKbound} hold for $|D_x^k L^n \mu_t|_L$.
\end{cor}
\begin{remark}
 In Theorem \ref{thm:HKbounds}, the assumption that $\mu_t$ is central doesn't play the role. However, the proof of Corollary \ref{cor:knHKbounds} heavily depends on the fact that $\mu_t$ is central (i.e., $L$ is bi-invariant).
\end{remark}

\subsection{\texorpdfstring{$L^1$}{L1} differentiability of the heat semigroup}

In this subsection, we discuss the $L^1$ differentiability of the heat semigroup and its consequences, which is partly taken from \cite{BSC2025}. We first write down the setup for the sake of completeness.

Let $(H_t)_{t>0}$ be $\mathcal C_0$ semigroup of operators acting on a Banach space $\mathcal X$. The semigroup $(H_t)_{t>0}$ is differentiable if, for all $f\in \mathcal X$ and all $t>0$, $H_tf$ is differentiable in $\mathcal X$. We will discuss the $L^1$ differentiability of the class of Markov semigroups associated with Gaussian convolution semigroups.

Consider now $(H_t)_{t>0}$ the self-adjoint Markov semigroup with infinitesimal generator $-L$ on $L^2(G,\nu)$, then the following properties hold
\begin{itemize}
    \item $(H_t)_{t>0}$ is a contraction on $L^p(G,\nu)$ for $1\le p\le \infty$;
    \item $(H_t)_{t>0}$ is bounded analytic on $L^p(G,\nu)$ for $1< p< \infty$ and satisfies
    \begin{equation}\label{eq:Lp-analytic}
    t\left\|\frac{\partial}{\partial t}H_tf\right\|_{p} \le  p^*\|f\|_p, \quad \forall t>0,
    \end{equation}
    where we recall  $p^*=\max\{p,\frac{p}{p-1}\}=\frac{2}{1-|(2/p)-1|}$.
\end{itemize}

Note that $t\frac{\partial}{\partial t}H_t=-tLH_t$, then \eqref{eq:Lp-analytic} is equivalent to the fact $t\|LH_t\|_{p\to p}\le p^*$. In general, we have for any  $n\in \mathbb N$
\begin{equation}\label{eq:analyticity-n}
t^n\|L^nH_t\|_{p\to p}\le (p^*)^n.
\end{equation}


We say that $H_t$ is ultracontractive if 
\begin{equation}\label{eq:ultrac}
\|H_t\|_{1\to \infty} \le  e^{M_0(t)}<\infty, \quad \forall t>0.    
\end{equation}
In this case, $H_t$ is given by a bounded kernel $h_t(x,y)$ (i.e., $\mu_t$  has continuous density with respect to the Haar measure $\nu$ so $h_t(x,y)=\mu_t(x^{-1}y)$ ) and 
\begin{equation}\label{eq:M0}
M_0(t)=\log\left(\sup_{x,y\in G} h_t(x,y)\right).   
\end{equation}

\begin{thm}[\protect{\cite[Theorem 4.1]{BSC2025}}]\label{thm:L1-diff}
Let $(H_t)_{t>0}$ be a self-adjoint Markov semigroup on $L^2(G,\nu)$ which is ultracontractive. Assume that $\nu$ is a probability measure. Then $(H_t)_{t>0}$ is $L^1$-differentiable and 
\[
t\left\|\frac{\partial}{\partial t}H_t\right\|_{1\to 1} \le 2e\max\{M_0(t/2), 2\}, \quad \forall t>0,
\]
where $M_0$ is defined by \eqref{eq:M0}.
\end{thm}

\begin{cor}
Referring to the preceding setting and notation, $(H_t)_{t>0}$ is also $L^\infty$-differentiable and 
$
t\left\|\frac{\partial}{\partial t}H_t\right\|_{\infty\to \infty} \le 2e\max\{M_0(t/2), 2\}, \forall t>0$.
 Moreover, for $p=1$ or $\infty$ there holds
\begin{equation}
\label{eq:Lp-diff-Lambda-2}
t^2\left\|\frac{\partial^2}{\partial t^2}H_t\right\|_{p\to p} \le 4e \left(\max\{M_0(t/2), 2\}\right)^2, \quad \forall t>0.   
\end{equation}
\end{cor}
\begin{proof}

The proof of the $L^\infty$-differentiability is similar as for the $L^1$ case and we specify the details here. Indeed, for $f\in L^{\infty}(G)$, the ultracontractivity of $H_t$ and the Riesz-Thorin interpolation theorem give that for any $1<p<\infty$
\begin{equation}\label{eq:ultra-Lp-Linfty}
\|H_t f\|_{\infty}\le  e^{M_0(t)/p}\|f\|_p, \quad \forall t>0.   
\end{equation}
Since the measure $\nu$ is normalized, it follows from the the $L^p$-analyticity of $H_t$ that
\[
t\left\|\frac{\partial}{\partial t}H_tf\right\|_{\infty}
=2t\|H_{t/2}\|_{p\to \infty}\left\|\frac{\partial}{\partial t}H_{t/2}f\right\|_{p}
\le  2p^* e^{M_0(t/2)/p}\|f\|_p\le  2p^* e^{M_0(t/2)/p}\|f\|_{\infty}.
\]
If $M_0(t/2)<2$, we pick $p=2$. If $M_0(t/2)>2$, we pick $p$ as a function of $t$ such that $p=M_0(t/2)$. We then have $p^*=p=M_0(t/2)$ and conclude that
\[
t\left\|\frac{\partial}{\partial t}H_t\right\|_{\infty\to \infty} \le 2e\max\{M_0(t/2), 2\}.
\]
Finally the same argument also applies to obtain the estimate in \eqref{eq:Lp-diff-Lambda-2} for $p=\infty$.

Next we write 
    \[
    \frac{\partial^2}{\partial t^2}H_tf=L^2H_{t/2}(H_{t/2}f).
    \]
Then in view of \eqref{eq:analyticity-n} and \eqref{eq:ultra-Lp-Linfty}, one has
\begin{align*}
t^2\left\|\frac{\partial^2}{\partial t^2}H_tf\right\|_{\infty} 
\le t^2\left\|H_{t/2}f\right\|_{p\to \infty} \left\|L^2H_{t/2}\right\|_{p}
\le 4(p^*)^2 e^{M_0(t/2)/p}\|f\|_\infty.
\end{align*}
Choosing the same $p$ as above, we obtain the desired bound \eqref{eq:Lp-diff-Lambda-2} for $p=\infty$.
\end{proof}

Under the assumption $(\mathrm{CK}\lambda)$ in Definition \ref{def:CK etc} (iii), the heat semigroup $H_t$ is ultracontractive with $M_0(t)=\frac{A}{t^{\lambda}}$, see Theorems \ref{thm:HKbounds} \rm{(ii)}.
Hence as a consequence of Theorem \ref{thm:L1-diff}, we have the following corollary. 
\begin{cor}\label{cor:L1-diff-Lambda}
Assume that $(\mu_t)_{t>0}$ is a symmetric Gaussian semigroup satisfying property $(\mathrm{CK}\lambda)$. Let $p=1$ or $\infty$. Then $(H_t)_{t>0}$ is $L^p$-differentiable and
\begin{equation}\label{eq:L1-diff-Lambda}
t\left\|\frac{\partial}{\partial t}H_t\right\|_{p\to p} \le C\max\Big\{\frac{2^{\lambda}A}{t^{\lambda}}, 2\Big\}, \quad \forall t>0;
\end{equation}
\begin{equation}\label{eq:L1-diff-Lambda-2}
t^2\left\|\frac{\partial^2}{\partial t^2}H_t\right\|_{p\to p} \le C \left(\max\Big\{\frac{2^{\lambda}A}{t^{\lambda}}, 2\Big\}\right)^2, \quad \forall t>0.   
\end{equation}

\end{cor}

\begin{remark}
    Under the same assumption, we will see later in Theorem \ref{thm:L1-Linfty-gradient} that the heat semigroup also satisfies an $L^1$ ``gradient'' bound in terms of the square root of the carr\'e du champs. 
\end{remark}



\section{Dimension-free \texorpdfstring{$L^p$}{Lp} bounds of Riesz transforms}
In this section, we study the sharp $L^p$ bounds (or dimension-free explicit $L^p$ bounds) of Riesz transforms: 
\begin{enumerate}[leftmargin=1.8em,label={\rm (\arabic*)}]
\item[{\rm (1)}]  Riesz transforms in directions $ X_i$, $i\in I$, i.e., $R_{i}= X_iL^{-1/2}$;
\item[{\rm (2)}]  Full Riesz transform  $R^G=\nabla L^{-1/2}=(R_1, R_2, \cdots)$, where $\nabla=(X_1, X_2, \cdots)$. Then 
\[
\|R^Gf\|_p=\||DL^{-1/2}f|_L\|_p, 
\]
where we recall $|DL^{-1/2}f|_L^2=\Gamma(L^{-1/2}f, L^{-1/2}f)=\sum_{i\in I} (R_if)^2$;
\item[{\rm (3)}] Second order Riesz transforms $R_iR_j$, for any $i,j\in I$.
\end{enumerate}
We point out that the bi-invariance of $L$ is essential in this section. In fact, such property will be heavily used in the proofs later.


\medskip

In order to apply the martingale transform techniques, we first introduce the following probabilistic setup and then establish several useful lemmas.

\medskip
\noindent{\bf Probabilistic setup.} Suppose that $(\Omega, \mathcal F, \mathbf P)$ is a complete probability space, filtered by $\mathcal F=\{\mathcal F_t\}_{t\ge 0}$, a family of right continuous sub-$\sigma$-fields of $\mathcal F$. Assume that $\mathcal F_0$ contains all the events of probability $0$. Let $M$ and $N$ be adapted, real-valued martingales which have right-continuous paths with left-limits (r.c.l.l.).  The martingale $N$ is differentially subordinate to $M$ if $|N_0|\le |M_0|$ and $\ang{M}_t-\ang{N}_t$ is a non-decreasing and nonnegative function of $t$.  The martingales $M_t$ and $N_t$ are said to be orthogonal if  the covariation process $\ang{M, N}_t=0$ for all $t$. 

\begin{lem}[\protect{\cite[Theorem 2.1.1]{Banuelos}}]\label{lem:MI}
Let $M$ and $N$ be two $\ell^2$-valued martingales with continuous paths such that $N$ is differentially subordinate to $M$. For $1<p<\infty$, set $p^*=\max\{p,\frac{p}{p-1}\}$. Then 
\[
\|N\|_p\le (p^*-1)\|M\|_p.
\]
In particular, if $M$ is an $\mathbb R$-valued martingale and $N=(N_1, N_2, \cdots)$ is an $\ell^2$-valued martingale such that $N_i$ is differentially subordinate to $M$ and $\sum_i\langle N_i\rangle_t\le \langle M\rangle_t$ for all $t>0$. Then $\|N\|_p\le (p^*-1)\|M\|_p$.
\end{lem}

Let $\beta_t$ be a one-dimensional Brownian motion on $\mathbb R$ (independent of $B_t$) with generator $\frac{\partial^2}{\partial y^2}$ starting from $\beta_0=y>0$. We set the hitting time 
\[
\tau=\inf\{t>0:\beta_t=0\}
\]
and denote by $\lambda_y$ the distribution of $\tau$. Then the harmonic extension is the subordination semigroup defined by 
\[
Qf(y,x)=e^{-y\sqrt{L}}f(x)=\int_0^{\infty} H_tf(x)d\lambda_y(t), \quad f\in\mathcal B(G).
\]

\begin{lem}\label{lem:Variation}
Let $M_t^f\coloneqq Qf(\beta_{t\wedge \tau},B_{t\wedge \tau})$. Then $M_t^f$ is a square integrable martingale and the quadratic variation is 
\[
\langle M^f\rangle_t=2\int_0^{t\wedge \tau} \left(\frac{\partial}{\partial y} Qf(\beta_s,B_s) \right)^2 ds+
2\int_0^{t\wedge \tau} \sum_{i\in I}\left(X_i Qf(\beta_s,B_s) \right)^2 ds.
\]
\end{lem}
\begin{proof}
We note that for any $x\in G$ and $u\ge 0$,
\begin{align*}
Qf(u,x)&=\int_0^{\infty} H_tf(x)d\lambda_u(t)=\int_0^{\infty} \mathbf E(f(B_t)|B_0=x) d\lambda_u(t)
\\&=\mathbf E(f(B_{\tau})|B_0=x,\beta_0=u).
\end{align*}
Since $M_{\tau}^f=Qf(0,B_{\tau})=f(B_\tau)$, then for any $0\le s\le \tau$ the above equation gives that
\[
\mathbf E(M_\tau^f|\mathcal F_s)=\mathbf E(f(B_{\tau})|\mathcal F_s)=\mathbf E(f(B_{\tau})|(\beta_s,B_s))=Qf(\beta_s,B_s).
\]
The quadratic variation is obtained from It\^o's lemma, see for instance \cite[Theorem 5.2, Section 5 in Chapter III]{Kuo} (or from Fukushima's decomposition theorem as in \cite[Theorem 5.2.3]{FOT}). Indeed, since $f\in \mathcal B(G)$, there exist $\phi \in\mathcal C^{\infty}(G_{\alpha})$ and $\alpha \in \aleph$ such that $f=\phi\circ \pi_{\alpha}$. Then for any $0<t\le \tau$, one has
\begin{align*}
Qf(\beta_t,B_t)
&=Q^{\alpha}\phi(\beta_t,B_t^{\alpha})
\\ &=Q^{\alpha}\phi(\beta_0,B_0^{\alpha})+\int_0^t (\nabla_{\alpha} Q^{\alpha}\phi(\beta_s,B_s^{\alpha}),dB_s^{\alpha} )+\int_0^t \frac{\partial}{\partial y}Q^{\alpha}\phi(\beta_s,B_s^{\alpha})d\beta_s,    
\end{align*}
where $(B_t^{\alpha})_{t\ge 0}$ is a Brownian motion on $G_{\alpha}$, $\nabla_{\alpha}=(X_{\alpha,1}, X_{\alpha,2}, \cdots, X_{\alpha,N_{\alpha}})$ for $N_{\alpha}=\#I_{\alpha}$ (recall $X_{\alpha,i}=d\pi_{\alpha}(X_i)$), and $Q^{\alpha}$ is harmonic extension of $H_t^{\alpha}$ on $G_{\alpha}$.
It follows that 
\begin{align*}
\langle M^f\rangle_t=\langle M^{\phi\circ \pi_{\alpha}}\rangle_t=
2\int_0^{t\wedge \tau} \sum_{i\in I}\left(X_{\alpha,i} Q^{\alpha}\phi(\beta_s,B_s^{\alpha}) \right)^2 ds+2\int_0^{t\wedge \tau} \left(\frac{\partial}{\partial y} Q^{\alpha}\phi(\beta_s,B_s^{\alpha}) \right)^2 ds.
\end{align*}
\end{proof}

\subsection{\texorpdfstring{$L^p$}{Lp} bounds of first and second order Riesz transforms}
Now we will give a martingale transform representation of directional Riesz transforms. For the finite dimensional case, the proof can be found in \cite{BBC}.
\begin{lem}\label{lem:RTrep}
For $f\in \mathcal B(G)$, define for each $i\in I$ the projection operator
\[
T_i f(x)=\lim_{y_0\to +\infty}\mathbf E\left(\int_0^{\tau} X_i Qf(\beta_s,B_s)d\beta_s\mid\beta_0=y_0,B_{\tau}=x\right).
\]
Then $R_if(x)=-2T_i f(x)$.
\end{lem}
\begin{proof}
Let $h\in \mathcal B(G)$, then $h(B_\tau)=Qh(\beta_\tau,B_\tau)=M_{\tau}^h$. For any $y_0>0$,
by It\^o's formula and the It\^o isometry we have
\begin{align*}
    &\int_{G}h(x)\,\mathbf E\left(\int_0^{\tau} X_i Qf(\beta_s,B_s)d\beta_s\mid\beta_0=y_0,B_{\tau}=x\right)d\nu(x)
    \\ &\quad\quad
    =\mathbf E_{y_0}\left(h(B_\tau)\int_0^{\tau} X_i Qf(\beta_s,B_s)d\beta_s\right)
    \\ &\quad\quad
    =2\int_{G}\int_0^{\infty} (y\wedge y_0)\frac{\partial}{\partial y} Qh(y,x)\, X_iQf(y,x)dyd\nu(x),
\end{align*}
where the last equality follows from the facts that the Green function of the Brownian motion $\beta_t$ starting at $y_0>0$ killed at 0 is $y\wedge y_0$ and $B_t$ is distributed according to $\nu$.

Since $L$ is bi-invariant, then $X_i$'s commute with the operator $L$ and the Poisson semigroup. Hence one has
\begin{align*}
&\int_{G}\int_0^{\infty} y\frac{\partial}{\partial y} Qh(y,x)\,  X_iQf(y,x)dyd\nu(x)
\\&\quad\quad=-\int_{G}h(x)\int_0^{\infty} y\sqrt{L} X_i e^{-2y\sqrt{L}} f(x)dy d\nu(x)
\\ &\quad\quad= -\frac14\int_{G}h(x) X_i L^{-1/2}f(x) d\nu(x).
\end{align*}
We conclude that 
\[
\int_{G}h(x) T_i f(x)d\nu(x)= -\frac12\int_{G}h(x)  X_i L^{-1/2}f(x) d\nu(x).
\]
\end{proof}

With above preparations, we state our main result in this section. 
\begin{thm}\label{thm:LpRT}
Let $1<p<\infty$. Then for every $j\in I$,
\[
\|R_jf\|_p\le 2(p^*-1)\|f\|_p, \quad f\in \mathcal B(G).
\]
Furthermore, $\|R^{G}f\|_p\le 2(p^*-1)\|f\|_p$.
\end{thm}
\begin{proof}
Since $\mathcal B(G)$ is a dense subset in $L^p(G)$, it suffices to prove for $f\in \mathcal B(G)$. We note the fact that the conditional expectation is a contraction in $L^p$ for $1<p<\infty$. Hence from Lemma \ref{lem:RTrep}, it is enough to consider the stochastic integral  
\[
N_t^i\coloneqq\int_0^{t\wedge\tau}X_i Qf(\beta_s,B_s)d\beta_s, \quad i\in I.
\]
Observe that $N_t^i$ is differentially subordinate to $M_t^f$, and moreover by Lemma \ref{lem:Variation},
\[
\sum_{i\in I}\langle N^i\rangle_t \le \langle M\rangle_t.
\] 
Applying martingale inequalities in Lemma \ref{lem:MI}, we obtain that $\|R_jf\|_p\le 2(p^*-1)\|f\|_p$ and $\|R^Gf\|_p\le 2(p^*-1)\|f\|_p$.
\end{proof}

\begin{remark} \label{rem:Rmn-to-0}
\

\begin{enumerate}
\item We do not pursue sharper $L^p$ bounds of Riesz transforms here. In fact, dimension-free bounds are enough for later applications. 

\item For any $n,m\in I$, $n>m$, denote $\overrightarrow{R}_{mn}=(0,\cdots, R_m, R_{m+1},\cdots, R_n, 0, \cdots)$. 
For $f\in \mathcal B(G)$, by definition one writes $f=\varphi\circ \pi_{\alpha}$ for some $\alpha\in \aleph$ and $\varphi\in C^{\infty}(G_{\alpha})$. Then if $m\notin I_{\alpha}$, one has $R_if=0$ for any $i\ge m$ and thus $\overrightarrow{R}_{mn}$ is a zero vector and hence
\begin{equation*}\label{eq:Rmn-limit}
\left\|\overrightarrow{R}_{mn} \varphi\right\|_p\longrightarrow 0, \quad \text{ as } m, n\to \infty.
\end{equation*} 
For $f\in L^p(G)$, let $\{f_i\}$ be a sequence from $\mathcal B(G)$ converging to $f$. Then
\[
\left\|\overrightarrow{R}_{mn} f\right\|_p\le \left\|\overrightarrow{R}_{mn} (f-f_i)\right\|_p+\left\|\overrightarrow{R}_{mn} f_i\right\|_p 
\longrightarrow 0, \quad \text{ as } m, n\to \infty \text{ and }i\to \infty.
\]

More generally, if $\mathcal B(G)$ is dense in any Banach space $(\mathbf B, \|\cdot\|_{\mathbf B})$, then the above conclusion also holds. That is, $\left\|\overrightarrow{R}_{mn} f\right\|_{\mathbf B}\to 0$ as $m,n\to \infty$.
\end{enumerate}
\end{remark}


Similarly, we can prove dimension-free bounds for second order Riesz transforms.
\begin{thm}\label{thm:Lp2ndRT}
Let $1<p<\infty$. Then for every $i,j\in I$,
\[
\|R_iR_jf\|_p\le 2(p^*-1)\|f\|_p, \quad f\in \mathcal B(G).
\]
\end{thm}
\begin{proof}
Similarly as the proof of Theorem \ref{thm:LpRT}, we proceed in two steps. 

\

\noindent{\bf Step 1:} Let $i,j\in I$. For any $f\in \mathcal B(G)$, we claim  that 
\begin{equation}\label{eq:2ndRiesz}
R_iR_jf(x)=-2S_{i,j} f(x),    
\end{equation}
where  $S_{i,j}$ is a projection operator defined by
\begin{equation}\label{eq:Sij}
S_{i,j} f(x)=\lim_{y_0\to +\infty}\mathbf E\left(\int_0^{\tau} A_{ij}\Big(\frac{\partial}{\partial y},\nabla\Big)^{\intercal} Qf(\beta_s,B_s)\cdot(d\beta_s,dB_s)|\beta_0=y_0,B_{\tau}=x\right).
\end{equation}
Here $A_{ij}$ is an infinite matrix with the entry $a_{i+1,j+1}=1$ and otherwise 0.

Let $h\in \mathcal B(G)$, then $h(B_\tau)=Qh(B_\tau,\beta_\tau)=M_{\tau}^h$. 
By It\^o's formula and the It\^o isometry, 
\begin{align*}
    &\int_{G}h(x)\,\mathbf E\left(\int_0^{\tau} A_{ij} \Big(\frac{\partial}{\partial y},\nabla\Big)^{\intercal}Qf(\beta_s,B_s)\cdot(d\beta_s,dB_s)| \beta_0=y_0,B_{\tau}=x\right)d\nu(x)
    \\ &\quad\quad
    = \mathbf E_{y_0}\left(h(B_\tau)\int_0^{\tau}A_{ij} \Big(\frac{\partial}{\partial y},\nabla\Big)^{\intercal}Qf(\beta_s,B_s)\cdot(d\beta_s,dB_s)\right)
    \\ &\quad\quad
    =2\int_{G}\int_0^{\infty} (y\wedge y_0)X_i Qh(y,x)\, X_jQf(y,x)dyd\nu(x).
\end{align*}
Note that $X_i$'s commute with the operator $L$ and the Poisson semigroup, then
\begin{align*}
&\int_{G}\int_0^{\infty} y X_i Qh(y,x)\,  X_jQf(y,x)dyd\nu(x)
\\&\quad\quad=-\int_{G}h(x)\int_0^{\infty} y X_iX_j e^{-2y\sqrt{L}} f(x)dy d\nu(x)
\\ &\quad\quad= -\frac14\int_{G}h(x) X_iX_j L^{-1}f(x) d\nu(x).
\end{align*}
We conclude that 
\[
\int_{G}h(x) S_{ij} f(x)d\nu(x)= -\frac12\int_{G}h(x)  X_iX_j L^{-1}f(x) d\nu(x).
\]

\

\noindent{\bf Step 2:}
Denote
\[
N_t^{ij}\coloneqq\int_0^{t\wedge\tau}A_{ij}\Big(\frac{\partial}{\partial y},\nabla\Big)^{\intercal}Qf(\beta_s,B_s)\cdot(d\beta_s,dB_s), \quad i,j\in I.
\]
By Lemma \ref{lem:Variation}, we note that $N_t^{ij}$ is differentially subordinate to $M_t^f$, i.e.,
$\langle N^{ij}\rangle_t \le \langle M\rangle_t$.
Now applying Lemma \ref{lem:MI}, one obtains  $\|R_iR_jf\|_p\le 2(p^*-1)\|f\|_p$.
\end{proof}
\begin{remark}
    Under the same assumption, we will see later in Theorem \ref{thm:L1-Linfty-gradient} that the heat semigroup also satisfies an $L^1$ ``gradient'' bound in terms of the square root of the carr\'e du champs. Besides, we note that by duality, Corollary \ref{cor:L1-diff-Lambda} infers the $L^{\infty}$-differentiability of  $(H_t)_{t>0}$ with the same $L^{\infty} \to L^{\infty}$ bounds. 
\end{remark}

\subsection{Application on gradient bounds of the heat semigroup}
The dimension-free bounds of Riesz transforms play an important role throughout the paper. As an illustration of their applications, we introduce the following results on the  $L^p$ gradient bounds of the heat semigroup.

\begin{prop}\label{prop:gradient-Lp}
Let $1<p<\infty$. Then there exists a constant $K>0$ (independent of $p$) such that for any $f\in \mathcal B(G)$, 
 \begin{equation}\label{eq:gradient-Lp}
 \||DH_tf|_L\|_{p} \le  2 K\sqrt{p^*}(p^*-1)t^{-1/2}\|f\|_p,    
 \end{equation}
where  $|Df|_L=|\Gamma(f,f)|^{\frac12}$ is defined in \eqref{eq:Dn}.
\end{prop}
\begin{proof}
Let $f\in \mathcal B(G)$. We apply the $L^p$ boundedness of the vector Riesz transform to $H_tf$, then
\[
\||DH_tf|_L\|_{p} \le 2(p^*-1)\left\|\sqrt{L} H_t f\right\|_p .
\]
Next the multiplicative inequality (see \cite[Section 2.6]{Pazy}) implies that
\begin{equation}\label{eq:multip-ineq}
\left\|\sqrt{L} H_t f\right\|_p \le K \left\|L H_t f\right\|_p^{\frac12} \left\| H_t f\right\|_p^{\frac12},
\end{equation}
where the constant $K$ doesn't depend on $p$. Since $H_t$ is a symmetric Markov semigroup, it is bounded analytic on $L^p$ with 
\[
t\|LH_tf\|_{p} \le  p^*\|f\|_p.
\]
We conclude the proof by combining the above estimates.
\end{proof}
\begin{thm}\label{thm:L1-Linfty-gradient}
Assume that $H_t$ is ultracontractive as in \eqref{eq:ultrac}. Let $p=1$ or $\infty$. Then for any $f\in \mathcal B(G)$
    \[
    \||DH_tf|_L\|_{p}\le C t^{-1/2}\max\{M_0(t)^{\frac32}, 1\}\|f\|_p.
    \]
In particular, if $(\mu_t)_{t>0}$ satisfies $(\mathrm{CK}\lambda)$, then
  \[
    \||DH_tf|_L\|_{p}\le C t^{-1/2}\max\{t^{-\frac32\lambda}, 1\}\|f\|_p.
  \]
\end{thm}
\begin{proof}
We use similar argument as in the proof of Theorem \ref{thm:L1-diff}. For $f\in \mathcal B(G)$,
let us start the estimate of $\||DH_tf|_L\|_{1}$ by using Jensen's inequality and semigroup property. Then for any $q\in (1,\infty)$
\[
\||DH_tf|_L\|_{1}\le  \||DH_tf|_L\|_{q}=\||DH_{t/2}H_{t/2}f|_L\|_{q} \le \||DH_{t/2}|_L\|_{q\to q}\|H_{t/2}f\|_{q}.
\]
Since $H_t$ is ultracontractivity, we apply the Riesz-Thorin interpolation theorem and obtain 
\begin{equation}\label{eq:ultracontractive}
\|H_{t/2} f\|_{ q}\le  e^{M_0(t/2)(1-\frac1q)}\|f\|_1, \quad \forall t>0.
\end{equation}
Hence it follows from  Proposition \ref{prop:gradient-Lp} that for $0<t<2$
\[
\||DH_tf|_L\|_{1} 
\le  2K\sqrt{q^*}(q^*-1)\left(\frac{t}{2}\right)^{-1/2} e^{M_0(t/2)(1-\frac1q)}\|f\|_1.
\]
We pick $q$ in the same way as the proof of Theorem \ref{thm:L1-diff}. That is, take $q=2$ if $M_0(t/2)<2$, and otherwise take $q$ as a function of $t$ such that $1-\frac1q=\frac1{M_0(t/2)}$ and $q'=M_0(t/2)$. In this case,  
\[
\sqrt{q^*}(q^*-1) \simeq \{M_0(t/2)^{\frac32},1\}.
\]
Then we obtain  
\[
\||DH_t|_L\|_{1\to 1}\le Ct^{-\frac12} \max\{M_0(t)^{\frac32}, 1\} .
\]
Finally, if $(\mu_t)_{t>0}$ satisfies $(\mathrm{CK}\lambda)$, then $H_t$ is ultracontractive with $M_0(t)=\frac{A}{t^{\lambda}}$ and the conclusion easily follows.

Next we move to the proof of the $L^{\infty}$ case. For $f\in \mathcal B(G)$, recall that 
\[
|DH_tf|_L\le C H_{t}|Df|_L,
\]
see for instance \cite[Theorem 3.3.18]{BGL2014}. It then follows from \eqref{eq:ultra-Lp-Linfty} and Proposition \ref{prop:gradient-Lp} that 
\begin{align*}
\||DH_tf|_L\|_{\infty} &\le  \|H_{t/2}|DH_{t/2}f|_L\|_{\infty}\le \|H_{t/2}\|_{q\to \infty}\||DH_{t/2}f|_L\|_{q} 
\\ &\le 2K\sqrt{q^*}(q^*-1)\left(\frac{t}{2}\right)^{-1/2} e^{M_0(t/2)/q}\|f\|_{q}.
\end{align*}
Note that $\nu$ is normalized and hence $\|f\|_{q}\le\|f\|_{\infty}$. Also, we pick $q$ properly, as in the case $p=1$. More precisely, take $q=2$ if $M_0(t/2)<2$, and otherwise take $q=M_0(t/2)$ so $q^*=q$. Then we conclude the estimate. 
\end{proof}

\section{Lipschitz spaces}
In this section, we introduce two types of Lipschitz spaces, whose seminorms are defined through the Markov semigroup and the (intrinsic) distance respectively. These spaces play an important role in studying the regularity problem for the Poisson equation in Section 6.

\subsection{Lipschitz spaces via the heat semigroup}
Recall that, unless otherwise stated, we are under the framework of Assumption \ref{assump} so the associated Markov semigroup $(H_t)_{t>0}$ is contractive on $L^p(G,\nu)$ for $1\le p\le \infty$ and bounded analytic on $L^p(G,\nu)$ for $1<p<\infty$.
\begin{defn}\label{def:Lip-Lambda}
Let  $n\in \mathbb N$ and $0<\theta<2n$. Let $1\le p\le \infty$. For $f\in \mathcal B(G)$,  the Lipschitz norm  $\Lambda_{\theta,n}^p(f)$ associated with the heat semigroup is given by
 \begin{equation}\label{eq:Lip-sg-p}
 \Lambda_{\theta,n}^p(f)= \sup_{t>0} t^{n-\frac\theta2}\left\|\frac{\partial^n}{\partial t^n} H_tf\right\|_p.
\end{equation}
In particular, we denote $ \Lambda_{\theta,1}^p(f)$ by $\Lambda_{\theta}^p(f)$ for $0<\theta<1$.

Let $\eta>0$ and $0<\theta<2\eta$. Let $1\le p\le \infty$. For $f\in \mathcal B(G)$,  the Lipschitz norm  $\Lambda_{\theta,\eta}^p(f)$ associated with the heat semigroup is given by
 \begin{equation}\label{eq:Lip-sg-p1}
 \Lambda_{\theta,\eta}^p(f)= \sup_{t>0} t^{\eta-\frac{\theta}2}\left\|L^{\eta} H_tf\right\|_p.
\end{equation}
\end{defn}
\begin{remark}\label{rem:Lip-sg-p-large}
\

\begin{enumerate}
\item For $1<p<\infty$, it follows from the $L^p$ analyticity of the heat semigroup that 
\[
t^{n-\frac\theta2}\left\|\frac{\partial^n}{\partial t^n} H_tf\right\|_p \le C  t^{-\frac\theta2}\|f\|_p, \quad t\ge1.
\]
The above estimate applies also to $p=1$ and $\infty$ under the assumption $(\mathrm{CK\lambda})$, thanks to \eqref{eq:L1-diff-Lambda} in Corollary \ref{cor:L1-diff-Lambda}. Hence, in order to show $\Lambda_{\theta,n}^p(f)<\infty$, it suffices to justify 
\[
\sup_{0<t<1}  t^{n-\frac\theta2}\left\|\frac{\partial^n}{\partial t^n} H_tf\right\|_p<\infty.
\]

\item Let $1\le p<\infty$. Since the measure $\nu$ is normalized, we have $\Lambda_{\theta,n}^p(f)\le \Lambda_{\theta,n}^{\infty}(f)$.

\item Let $0<\theta<\beta<2n$. We have $ \Lambda_{\theta,n}^p(f)\le \Lambda_{\beta,n}^p(f)$ for $1\le p\le \infty$.
Indeed, assume that $\Lambda_{\beta,n}^p(f)<\infty$. Then for $0<t<1$,
\[
 t^{n-\frac\theta2}\left\|\frac{\partial^n}{\partial t^n} H_tf\right\|_p = t^{\frac{\beta-\theta}{2}} t^{n-\frac\beta2}\left\|\frac{\partial^n}{\partial t^n} H_tf\right\|_p
 \le t^{n-\frac\beta2}\left\|\frac{\partial^n}{\partial t^n} H_tf\right\|_p<\infty
\]
and we conclude from the first remark above.
\end{enumerate}
\end{remark}

\begin{lem}\label{lem:equiv-Lip-sg}
Suppose that $\theta>0$ and $\eta, \tau>\theta/2$.
Let $1< p<\infty$, then  $\Lambda_{\theta, \eta}^p(f)<\infty$ if and only if $\Lambda_{\theta, \tau}^p(f)<\infty$. 
\end{lem}

\begin{proof}
Let $n$ be the smallest integer larger than $\theta/2$. It is enough to prove that, for any $\eta>\theta/2$ ($\eta\neq n$), $\Lambda_{\theta, \eta}^p(f)<\infty$ if and only if $\Lambda_{\theta, n}^p(f)<\infty$. In the following, we will divide the proof into three steps. 
\medskip

\noindent{\bf Step 1: $\eta$ is an integer.} It suffices to prove for $\eta=n+1$. Suppose that $\Lambda_{\theta,n}^p(f)<\infty$. The $L^p$ analyticity of the heat semigroup gives that
\begin{equation}\label{eq:m-less-n}
\begin{split}
t^{n+1-\frac\theta2}\left\|\frac{\partial^{n+1}}{\partial t^{n+1}} H_tf\right\|_p
&= t^{n-\frac\theta2}t\left\|LH_{t/2} L^nH_{t/2}f\right\|_p 
\\ &\le 2p^*t^{n-\frac\theta2} \left\|L^n H_{t/2}f\right\|_p \le 2^{n+1-\frac\theta2}p^*\Lambda_{\theta,n}^p(f).      
\end{split}
\end{equation}
Thus one has $\Lambda_{\theta,n+1}^p(f)<\infty$ by taking the supremum over $t>0$ on the left hand side.

Conversely, assume that $\Lambda_{\theta,n+1}^p(f)<\infty$. Since the heat semigroup converges to equilibrium, we have that
\begin{equation}\label{eq:convgence}
\frac{\partial^n}{\partial t^n} H_tf=H_t(-L)^nf\longrightarrow \int_G (-L)^nfd\nu =0, \quad \text{ as }t\to \infty
\end{equation}
holds pointwise.
Hence one can write
\begin{equation}\label{eq:time-D}
\frac{\partial^n}{\partial t^n} H_tf=-\int_{t}^{\infty} \frac{\partial^{n+1}}{\partial s^{n+1}} H_sf ds.    
\end{equation}
It follows  that
\begin{equation}\label{eq:n-less-m}
\begin{split}
t^{n-\frac{\theta}2}\left\|\frac{\partial^n}{\partial t^n} H_tf\right\|_p
&\le t^{n-\frac{\theta}2} \int_{t}^{\infty} \left\|\frac{\partial^{n+1}}{\partial s^{n+1}} H_sf\right\|_p ds
\\ &\le t^{n-\frac{\theta}2} \int_{t}^{\infty} s^{\frac{\theta}{2}-(n+1)}s^{n+1-\frac{\theta}{2}}\left\|\frac{\partial^{n+1}}{\partial s^{n+1}} H_sf\right\|_p ds <\infty
\end{split}
\end{equation}
and we conclude the proof.

\medskip

\noindent{\bf Step 2: $\eta$ is not an integer and $\eta <n$.}  We note that $n<\eta+1$. If $\Lambda_{\theta, \eta}^p(f)<\infty$, the $L^p$ analyticity of the heat semigroup implies that, for all $t>0$,
    \[
     t^{n-\frac\theta2}\left\|\frac{\partial^n}{\partial t^n}H_tf\right\|_p
     = t^{n-\frac\theta2}\left\|L^{n-\eta}H_{t/2}L^{\eta}H_{t/2}f\right\|_p
     \le C t^{\eta-\frac\theta2}\left\|L^{\eta}H_{t/2}f\right\|_p
     <\infty.
    \]
Hence one has $\Lambda_{\theta,n}^p(f)<\infty$.

Conversely, if $\Lambda_{\theta,n}^p(f)<\infty$, then the same argument as above deduces that
\(
\Lambda_{\theta,\eta+1}^p(f)<\infty.
\)
Similarly as in the proof of Step 1, we also have 
\[
L^{\eta} H_tf=-\int_{t}^{\infty} \frac{\partial}{\partial s}L^{\eta} H_sf ds.
\]
Therefore using the Minkowski inequality, one gets for every $t>0$
\begin{align*}
t^{\eta-\frac\theta2}\left\|L^{\eta}H_tf\right\|_p
&\le 
t^{\eta-\frac\theta2} \int_{t}^{\infty} \left\|\frac{\partial}{\partial s} L^{\eta} H_sf\right\|_p ds
\\ &\le 
t^{\eta-\frac\theta2} \int_{t}^{\infty} s^{\frac\theta2-\eta-1}s^{\eta+1-\frac\theta2}\left\|L^{\eta+1} H_sf\right\|_p ds 
\\&\le
C t^{\eta-\frac\theta2} \int_{t}^{\infty} s^{\frac\theta2-\eta-1}ds \,\sup_{r>0} r^{\eta+1-\frac\theta2}\left\|L^{\eta+1}H_rf\right\|_p.
\end{align*}
Since $\eta >\theta/2$, the above equation is finite and we thus show $\Lambda_{\theta,\eta}^p(f)<\infty$. 

\medskip

\noindent{\bf Step 3: $\eta$ is not an integer and $\eta >n$.} Observe that there exists $k\in \mathbb N$ such that $n+k>\eta$. If $\Lambda_{\theta,n}^p(f)<\infty$, then the $L^p$ analyticity of the heat semigroup gives that 
\[
\Lambda_{\theta,\eta}^p(f)\le C\Lambda_{\theta,n}^p(f)<\infty.
\]
On the other hand, assume $\Lambda_{\theta,\eta}^p(f)<\infty$. It follows from Step 1 and the $L^p$ analyticity of the heat semigroup that
\[
\Lambda_{\theta,n}^p(f) \le C\Lambda_{\theta,n+k}^p(f) \le C\Lambda_{\theta,\eta}^p(f)<\infty.
\]
The proof is therefore complete.
\end{proof}

\begin{defn}\label{def:Lip-sg}
Let $\theta>0$. Let $1< p< \infty$. Suppose that $n$ is the smallest integer strictly greater than $\theta/2$. We define the space $\Lambda_{\theta}^p(G)$ as the completion of the set of functions  $f\in \mathcal B(G)$ such that 
 \begin{equation}\label{eq:Lip-sg-p-finite}
 \Lambda_{\theta,n}^p(f)= \sup_{t>0} t^{n-\frac\theta2}\left\|\frac{\partial^n}{\partial t^n} H_tf\right\|_p<\infty
\end{equation}
with respect to the $\Lambda_{\theta}^p$ norm given by 
\[
\|f\|_{\Lambda_{\theta}^p}:=\|f\|_p+ \Lambda_{\theta,n}^p(f).
\]
For simplicity, we denote the seminorm $\Lambda_{\theta,n}^p(f)$ by $\Lambda_{\theta}^p(f)$.
\end{defn}

\begin{prop}\label{prop:Banach}
For $1<p<\infty$ and $\theta\ge 0$, $\Lambda_{\theta}^p(G)$ is a Banach space.
\end{prop}
\begin{proof}
Let $\{f_i\}_i$ be a sequence of Cauchy sequence in $\Lambda_{\theta}^p(G)$ and let $f$ be the $L^p$ limit of $f_i$. For any $n>\theta/2$,  the $L^p$ analyticity of the heat semigroup gives that
\[
t^{n}\left\|\frac{\partial^n}{\partial t^n} H_t(f-f_i)\right\|_p\le C\|f-f_i\|_p.
\]
Then by the triangle inequality, one has
$\lim_{i\to \infty}t^{n}\left\|\frac{\partial^n}{\partial t^n} H_tf_i\right\|_p=t^{n}\left\|\frac{\partial^n}{\partial t^n} H_tf\right\|_p$. Hence
\[
t^{n-\frac\theta2}\left\|\frac{\partial^n}{\partial t^n} H_tf\right\|_p=\lim_{i\to \infty}t^{n-\frac\theta2}\left\|\frac{\partial^n}{\partial t^n} H_tf_i\right\|_p\le \lim_{i\to \infty} \Lambda_{\theta,n}^p(f_i)<\infty.
\]
We conclude that $\Lambda_{\theta,n}^p(f)<\infty$ and $f\in \Lambda_{\theta}^p(G)$. Similarly, for each fixed positive integer $j$, there holds
\[
\Lambda_{\theta,n}^p(f-f_j) \le \lim_{i\to \infty} \Lambda_{\theta,n}^p(f_i-f_j).
\]
Taking the limit $j\to \infty$ and recalling that $\{f_i\}$ is Cauchy,  the proof is completed. 
\end{proof}

Now we discuss the analogues of Lemma \ref{lem:equiv-Lip-sg} for $p=1, \infty$. These cases are more complicated, due to the lack of $L^1$ and $L^{\infty}$ analyticity.
Recall the Lipschitz norms $\Lambda_{\theta,n}^{\infty}$ and $\Lambda_{\theta,n}^{1}$ in Definition \ref{def:Lip-Lambda}. 
Our main statement is as follows.

\begin{lem} \label{lem:Lip-sg-p-alter-1infty}
Assume that $\mu_t$ satisfies property $(\mathrm{CK\lambda})$. Let $\theta>0$ and $p=1$ or $\infty$. Let $n>\frac{\theta}{2}$ be an integer.  Suppose that $f\in \mathcal B(G)$.
\begin{enumerate}
    \item If $\Lambda_{\theta+2\lambda,n}^{p}(f)<\infty$, then $\Lambda_{\theta,n+1}^{p}(f)<\infty$.  In particular, if property $(\mathrm{CK}0^+)$ holds, then  $\Lambda_{\theta+\varepsilon,n}^{p}(f)<\infty$ for any $\varepsilon>0$ implies that $\Lambda_{\theta,n+1}^{p}(f)<\infty$.
    \item If $\Lambda_{\theta,n+1}^{p}(f)<\infty$, then $\Lambda_{\theta,n}^{p}(f)<\infty$.
\end{enumerate}
\end{lem}
\begin{proof}
For the proof of (i), consider first the case $p=\infty$
by assuming that $\Lambda_{\theta,n}^{\infty}(f)<\infty$. For any $1<q<\infty$, we observe from  \eqref{eq:ultra-Lp-Linfty} and the $L^q$ analyticity of $H_t$ in \eqref{eq:analyticity-n} that
\begin{align*}
t\left\|L^{n+1} H_tf\right\|_{\infty}
&=
t\left\|H_{t/2}L^{n+1} H_{t/2}f\right\|_{\infty}
\le t e^{M_0(t/2)/q} \left\|L^{n+1} H_{t/2}f\right\|_{q}
\\ &\le 4q^*e^{M_0(t/2)/q}\left\|L^n H_{t/2}f\right\|_{q}.    
\end{align*}
One takes $q>1$ as a function of $t\in(0,1)$ such that $M_0(t/2)/q\simeq1$ and $q^*\simeq t^{-\lambda}$. Since $\nu$ is normalized, then $\left\|L^n H_{t/2}f\right\|_{q}\le \left\|L^n H_{t/2}f\right\|_{\infty}$. We then have 
\[
\sup_{0<t<1} t^{n+1-\frac{\theta}{2}} \left\|L^{n+1} H_tf\right\|_{\infty} 
\le C\sup_{0<t<1} t^{n-\frac{\theta+2\lambda}{2}} \left\|L^n H_tf\right\|_{\infty}<\infty,
\]
where the constant $C$ is independent of the choice $q$. In view of Remark \ref{rem:Lip-sg-p-large} (i), we conclude $\Lambda_{\theta,n+1}^{\infty}(f)<\infty$.

Next considering the case $p=1$, we assume $\Lambda_{\theta,n}^1(f)<\infty$ and apply \eqref{eq:ultracontractive}, then
\begin{align*}
t\left\|L^{n+1} H_tf\right\|_{1}
&\le t\left\|L^{n+1} H_{t}f\right\|_{q}
\le q^* \left\|L^{n} H_{t}f\right\|_{q}
\\&\le q^*e^{M_0(t/2)(1-\frac1q)}\left\|L^{n} H_{t/2}f\right\|_{1}.
\end{align*}
Taking $q>1$ as a function of $t\in(0,1)$ such that $M_0(t)(1-\frac1q)=1$ and $q^*\simeq t^{-\lambda}$, we then conclude $\Lambda_{\theta,n+1}^1(f)<\infty$ in a similar way as for the case $p=\infty$.

It remains to prove (ii) and we also start with the case $p=\infty$. In light of \eqref{eq:convgence}, we still have $\frac{\partial^n}{\partial t^n}H_t f\to 0$ in $L^{\infty}$. Hence one can run the same proof in the second part of Step 1 for Lemma \ref{lem:equiv-Lip-sg}. The case $p=1$ follows also analogously as for $p=\infty$.

\end{proof}

\paragraph{Riesz transforms on Lipschitz spaces}
Now we summarize some results regarding the boundedness of Riesz transforms with respect to the Lipschitz seminorm $\Lambda_{\theta,n}^p(\cdot)$. Those results will be useful later in the study of  the regularity for the solution of the Poisson equations. 

\begin{prop}\label{prop:RTinLip}
Let $n\in \mathbb N$ and let $0<\theta<2n$. Consider $1<p<\infty$. Then the following hold for $f\in \mathcal B(G)$:
\begin{enumerate}
\item $\Lambda_{\theta,n}^p(R_i f)\le 2(p^*-1)\Lambda_{\theta,n}^p( f)$, for every $i\in I$.
\item $\Lambda_{\theta,n}^p(R_iR_j f)\le 2(p^*-1)\Lambda_{\theta,n}^p( f)$, for every $i,j\in I$.
\item $\Lambda_{\theta,n}^p\left(\overrightarrow R_{mn}f\right)\to 0$ as $m, n \to \infty$.
\end{enumerate}
\end{prop}
\begin{proof}
Assume first $1<p<\infty$. The two results (i) and (ii) are direct consequences of Theorems \ref{thm:LpRT} and \ref{thm:Lp2ndRT}. Item (iii) follows from Remark \ref{rem:Rmn-to-0} (ii) and the fact that $\Lambda_{\theta}^p(G)$ is Banach (see Proposition \ref{prop:Banach}).
\end{proof}

\begin{prop}\label{prop:RTinLip-infty}
Assume that $\mu_t$ satisfies property $(\mathrm{CK\lambda})$ for some $0<\lambda<1$. Let $n\in \mathbb N$ and let $0<\theta<2n$. Consider $p=1$ or $\infty$. Then the following hold for $f\in \mathcal B(G)$:
\begin{enumerate}
\item $\Lambda_{\theta,n}^p(R_i f)\le C\Lambda_{\theta+2\lambda,n}^p(f)$, for every $i\in I$. In particular, if property $(\mathrm{CK}0^+)$ is satisfied, then
$\Lambda_{\theta,n}^p(R_i f)\le C\Lambda_{\theta+\varepsilon,n}^p(f)$ for any $\varepsilon>0$.
\item $\Lambda_{\theta,n}^p(R_iR_j f)\le C\Lambda_{\theta+2\lambda,n}^p(f)$, for every $i,j\in I$. In particular, if property $(\mathrm{CK}0^+)$ is satisfied, then 
\(\Lambda_{\theta,n}^p(R_iR_j f)\le C\Lambda_{\theta+\varepsilon,n}^p(f).
\)
\end{enumerate}
\end{prop}
\begin{proof}
    We consider the case $p=\infty$. The argument is similar as in the proof of Lemma \ref{lem:Lip-sg-p-alter-1infty}. For any $1<q<\infty$ and $i\in I$, it follows from the ultracontractivity of $H_t$ in \eqref{eq:ultra-Lp-Linfty} and Theorem \ref{thm:LpRT} that 
    \[
    \left\|L^n H_tR_if\right\|_{\infty}
    \le e^{M_0(t/2)/q} \left\|L^n H_{t/2}R_if\right\|_{q}
    \le 4q^*e^{M_0(t/2)/q}\left\|L^n H_{t/2}f\right\|_{q}.   
    \]
    One takes $q>1$ as a function of $t\in(0,1)$ such that $M_0(t/2)/q\simeq1$ and $q^*\simeq t^{-\lambda}$. 
    We then have 
\[
\sup_{0<t<1} t^{n-\frac{\theta}{2}} \left\|L^n H_tR_if\right\|_{\infty} 
\le C\sup_{0<t<1} t^{n-\frac{\theta+2\lambda}{2}} \left\|L^n H_tf\right\|_{\infty}<\infty,
\]
 Remark \ref{rem:Lip-sg-p-large} (i) then gives that $\Lambda_{\theta,n}^{\infty}(R_jf)\le C\Lambda_{\theta+2\lambda,n}^{\infty}(f)$. 
 
 The proof for (ii) is almost the same verbatim, except that one applies  Theorem \ref{thm:Lp2ndRT} instead of Theorem \ref{thm:LpRT}.


\end{proof}

\subsection{Lipschitz spaces via distance}
In this subsection, we assume that $G$ is equipped with a left-invariant distance $d$. Later, this distance $d$ will often be the intrinsic distance associated with the Laplacian $L$ (see Definition \ref{def:distance}).

Set 
\[
\tau_hf(x)=f(xh), \quad\Delta_h^0f(x)=f(x).
\]
For any integer $k\ge 0$, we define 
\begin{equation}\label{eq:Delta-k}
\Delta_h^{k+1}f(x) =\Delta_h^k\tau_hf(x)- \Delta_h^kf(x).
\end{equation}
In particular, we write $\Delta_h f\coloneqq\Delta_h^1 f=f(xh)-f(x)$ for simplicity.

\begin{defn}
Let $k\in \mathbb N$ and $0<\theta<k$. Let $1\le p\le \infty$. 
For  $f\in \mathcal B(G)$, the Lipschitz norm $\mathrm L_{\theta,k}^p(f)$  associated with the  distance $d$ is given by
\begin{equation}\label{eq:Lip-dist-p}
\mathrm L_{\theta,k}^p(f)= \sup_{y\in G}\frac{\|\Delta_y^kf\|_p}{d(e,y)^{\theta}}.
\end{equation} 
\end{defn}

\begin{remark}\label{rem:Lip-dis}
 Similarly as in Remark \ref{rem:Lip-sg-p-large} (iii),  let any $0<\theta<\beta<k$, then we have that $\mathrm L_{\theta,k}^p(f)\le L_{\beta,k}^p(f)$ for any $1\le p\le \infty$.
\end{remark}

\begin{lem}
Let $1\le p\le \infty$. For any $0<\theta<k\le \ell$, there exist constants $c,C>0$ depending on $\theta, k$ and $\ell$ such that for any $f\in \mathcal B(G)$,
\[
c\,\mathrm L_{\theta,\ell}^p(f)\le \mathrm L_{\theta,k}^p(f)\le C\,\mathrm L_{\theta,\ell}^p(f).
\]
\end{lem} 
\begin{proof} For any fixed $k\in \mathbb N$, it suffices to consider the case $\ell=k+1$. We drop the reference to $\theta$ and $p$ which are fixed. Obviously, since $\Delta_h^{k+1}f(x) =\Delta_h^k\tau_hf(x)- \Delta_h^kf(x)$, it follows that $\mathrm L_{k+1}(f)\le 2\mathrm L_{k}(f)$.
 
For the other direction, note that
\[
\Delta_{h^2}=\tau_h^2-I=(\tau_h-1)(\tau_h+1)=\Delta_h(2I+\Delta_h).
\]
Hence for $k\in \mathbb N$,
\[
\Delta^k_{h^2}=\Delta^k_h(2I+\Delta_h)^k= 2^k\Delta_h^k+ \Delta^{k+1}_h\sum_{j=0}^{k-1}\tau_h^j.
\]
It follows that
\[
2^k\Delta_{h}^kf(x)=  \Delta^k_{h^2}f(x) -  \sum_{j=0}^{k-1}\Delta_h^{k+1}\tau_{h^j}f(x).
\]
This gives
\[
2^k \|\Delta^k_hf\|_p\le \|\Delta^k_{h^2}f\|_p+ (k-1)\|\Delta_h^{k+1}f\|_p.
\] 
Dividing by $d(e,h)^\theta$ and using the fact that $d(e,h^2)\le 2d(e,h)$, we get
$$
2^k \frac{\|\Delta^k_hf\|_p}{d(e,h)^\theta}\le 2^\theta\frac{\|\Delta^k_{h^2}f\|_p}{(2d(e,h))^\theta}+ k\frac{\|\Delta_h^{k+1}f\|_p}{d(e,h)^\theta}\le 2^\theta\frac{\|\Delta^k_{h^2}f\|_p}{d(e,h^2)^\theta}+ k\frac{\|\Delta_h^{k+1}f\|_p}{d(e,h)^\theta} .
$$ 
Taking the supremum over $h\in G$, this gives
$$
(2^k-2^\theta)\,\mathrm L_k(f)\le k\,\mathrm L_{k+1}(f).
$$
The third author learned this argument from an unpublished manuscript of Carl Herz from around 1970.
\end{proof}

\begin{defn}\label{def:Lip-distance}
Let $\theta>0$. Let $1\le p\le \infty$. Suppose that $k$ is the smallest integer strictly greater than $\theta$. We define the space $\mathrm L_{\theta}^p(G)$ as the completion of the set of functions  $f\in \mathcal B(G)$ such that $\mathrm L_{\theta,k}^p(f)<\infty$
with respect to the $\mathrm L_{\theta}^p$ norm given by 
\[
\|f\|_{\mathrm L_{\theta}^p}:=\|f\|_p+ \mathrm L_{\theta,k}^p(f).
\]
For simplicity, we denote the seminorm $\mathrm L_{\theta,k}^p(f)$ by $ \mathrm L_{\theta}^p(f)$.
\end{defn}

\subsection{Comparison of Lipschitz seminorms: the case \texorpdfstring{$1<p<\infty$}{1p-infty}}
In this subsection, we compare the two seminorms $\mathrm L_{\theta}^p(f)$ and $\Lambda_{\beta}^p(f)$ for $1<p<\infty$ with appropriate choice of $\theta$ and $\beta$. Here in the seminorm $\mathrm L_{\theta}^p(f)$ the distance $d$  will be referred as the intrinsic distance associated to the Laplacian $L$.
In the following result, we assume $(\mathrm{CK}\lambda)$ for some $\lambda\in (0,1)$ and the proof also applies to $p=1, \infty$ with slight modification.

\begin{thm}\label{thm:LipSG<LipDist}
Let $1\le p\le \infty$ and let $\theta\in (0,2)$. Assume that $\mathrm L_{\theta}^p(f)<\infty$. 
\begin{enumerate}
    \item If $(\mu_t)_{t>0}$ satisfies $(\mathrm{CK}\lambda)$ with $0<\lambda<\frac{\theta}{\theta+2}$, then  $\Lambda_{\beta}^p(f)<\infty$ for $\beta\in (0,1)$ such that $\beta \le (1-\lambda)\theta-2\lambda$. In particular, if $(\mu_t)_{t>0}$ satisfies $(\mathrm{CK}0^+)$, then $\Lambda_{\beta}^p(f)<\infty$ for $0<\beta<\min\{\theta,1\}$.
    \item If $\theta>1$ and $(\mu_t)_{t>0}$ satisfies $(\mathrm{CK}\lambda)$ with $0<\lambda<\frac{\theta-1}{\theta+4}$, then $\Lambda_{\beta,2}^p(f)<\infty$ for $1\le \beta \le (1-\lambda)\theta-4\lambda$. In particular, if $(\mu_t)_{t>0}$ satisfies $(\mathrm{CK}0^+)$, then $\Lambda_{\beta,2}^p(f)<\infty$ for $1\le \beta <\theta$.
\end{enumerate}
\end{thm}
\begin{proof}
Some  ideas in the argument below come from \cite{SC1990, Stein}. 
Let $\mathrm L_{\theta}^p(f)<\infty$. We note that $\int_G \frac{\partial}{\partial t} h_t(x) d\mu(x)=0$. Then
\[
\frac{\partial}{\partial t} H_tf(x)=\frac12\int_G \frac{\partial}{\partial t} h_t(y) (f(xy)+f(xy^{-1})-2f(x))d\nu(y)
\]
and hence
\begin{align*}
\left\|\frac{\partial}{\partial t} H_tf\right\|_p
& 
\le \int_G\left|\frac{\partial}{\partial t}h_t(y)\right|\,\left\|\Delta_{y^{-1}}\Delta_yf\right\|_pd\nu(y)
\\ &
\le \int_G d(e,y)^{\theta} \left|\frac{\partial}{\partial t} h_t(y)\right|d\nu(y)\, \mathrm L_{\theta,2}^p(f).  
\end{align*}

\begin{enumerate}
\item Assume that $(\mu_t)_{t>0}$ satisfies $(\mathrm{CK}\lambda)$ with $0<\lambda<\frac{\theta}{\theta+2}$. Letting $0<\beta<1$, we aim to show that $\Lambda_{\beta}^p(f)<\infty$. In view of Remark \ref{rem:Lip-sg-p-large}, it suffices to justify that 
\begin{equation}\label{eq:sup-time-der}
\sup_{0<t<1} t^{1-\frac{\beta}{2}}\int_G d(e,y)^{\theta}\left|\frac{\partial}{\partial t}h_t(y)\right| d\nu(y)<\infty.    
\end{equation}

We divide the integral in \eqref{eq:sup-time-der} for $\{y:d(e,y)\le \zeta t^{\frac{1-\lambda}2}\}$ and $\{y:d(e,y)\ge \zeta t^{\frac{1-\lambda}2}\}$, where $\zeta$ is a constant to be determined later. By the $L^1$-differentiability  \eqref{eq:L1-diff-Lambda}, one has 
\begin{align*}
    t^{1-\frac{\beta}{2}}\int_{d(e,y)\le \zeta t^{\frac{1-\lambda}2}} d(e,y)^{\theta}\left|\frac{\partial}{\partial t}h_t(y)\right| d\nu(y)
    & \le C t^{\frac{(1-\lambda)\theta-\beta}2+1}\int_{G}\left|\frac{\partial}{\partial t}h_t(y)\right| d\nu(y)
    \\ &\le C t^{\frac{(1-\lambda)\theta-\beta}2-\lambda}<\infty, \quad \forall 0<t<1
\end{align*} 
provided $\frac{(1-\lambda)\theta-\beta}2-\lambda\ge 0$, i.e., $\beta\le(1-\lambda)\theta-2\lambda$. 

Next applying Theorem \ref{thm:HKbounds} (see \cite[Theorem 4.2]{BSC2002}), we obtain that for any $\alpha,\beta>0$,
\begin{align*}
    &t^{1-\frac{\beta}{2}}\int_{d(e,y)\ge \zeta t^{\frac{1-\lambda}2}} d(e,y)^{\theta} \left|\frac{\partial}{\partial t}h_t(y)\right|d\nu(y)
    \\ &\quad\le 
    Ct^{\frac{\theta-\beta}2}\int_{d(e,y)\ge \zeta t^{\frac{1-\lambda}2}} \left(\frac{d(e,y)^2}{t}\right)^{\theta/2}\exp\left(\frac{A}{t^{\lambda}}-\frac{d(e,y)^2}{Ct}\right)d\nu(y)
    \\ & \quad
    \le   Ct^{\frac{\theta-\beta}2}\exp\left(-\frac{A'}{t^{\lambda}}\right)<\infty, \quad \forall 0<t<1
\end{align*}
where for the second inequality we choose $\zeta$ large enough such that $\zeta^2/C-A=A'>0$. 

To summarize, \eqref{eq:sup-time-der} holds for $\beta\in (0,1)$ such that $\beta \le (1-\lambda)\theta-2\lambda$. Therefore,  there exists a constant $C>0$ such that $\Lambda_{\beta}^p(f)\le C\,\mathrm L_{\theta}^p(f)$ and we complete the proof.

\item Assume that $\alpha>1$ and $(\mu_t)_{t>0}$ satisfies $(\mathrm{CK}\lambda)$ with $0<\lambda<\frac{\theta-1}{\theta+4}$. Consider now $\beta\in [1,2)$. Similarly as case (i), we have
\[
\frac{\partial^2}{\partial t^2} H_tf(x)=\frac12\int_G \frac{\partial^2}{\partial t^2} h_t(y) (f(xy)+f(xy^{-1})-2f(x))d\nu(y).
\]
So it is enough to show that 
\[
\sup_{0<t<1} t^{2-\frac{\beta}{2}}\int_G d(e,y)^{\theta}\left|\frac{\partial^2}{\partial t^2}h_t(y)\right| d\nu(y)<\infty    
\]
by estimating two integrals over $\{y:d(e,y)\le \zeta t^{\frac{1-\lambda}2}\}$ and $\{y:d(e,y)\ge \zeta t^{\frac{1-\lambda}2}\}$. The latter follows exactly the same way as above. For the former, we apply Corollary \ref{cor:L1-diff-Lambda} and obtain that for $1\le \beta \le (1-\lambda)\theta-4\lambda$, 
\begin{align*}
    t^{2-\frac{\beta}{2}}\int_{d(e,y)\le \zeta t^{\frac{1-\lambda}2}} d(e,y)^{\theta}\left|\frac{\partial^2}{\partial t^2}h_t(y)\right| d\nu(y)
    & \le C t^{\frac{(1-\lambda)\theta-\beta}2+2}\int_{G}\left|\frac{\partial^2}{\partial t^2}h_t(y)\right| d\nu(y)
    \\ &\le C t^{\frac{(1-\lambda)\theta-\beta}2-2\lambda}<\infty, \quad \forall 0<t<1.
\end{align*} 
\end{enumerate}
\end{proof}


Now we turn to the proof of the reverse direction, which uses semigroup property and is not restricted to the assumption $(\mathrm{CK}\lambda)$. We start with the following $L^p$-Poincar\'e inequality.
\begin{lem}\label{lem:p-Poincare}
Let $1\le p< \infty$. For any $f\in L^p(G)\cap W^{1,p}(G)$ and $y\in G$, we have 
\[
\int_G |f(xy)-f(x)|^p d\nu(x) \le d(e,y)^p \int_G \Gamma(f,f)^{\frac p2}(z)d\nu(z).
\]
If $p=\infty$, then for $f\in  L^{\infty}(G)\cap \mathrm{Dom}(\mathcal E)$ and $y\in G$, we have 
\[
\|\Delta_yf\|_{\infty}\le d(e,y)\|\Gamma(f,f)^{\frac 12}\|_{\infty}.
\]
\end{lem}

\begin{proof}
Since the space of Bruhat test functions is dense in $L^p(G) \cap W^{1,p}(G)$, it suffices to prove the statement for any $f\in \mathcal B(G)$. Write  $f=\phi_\alpha\circ \pi_\alpha$,
where $\phi_\alpha$ is smooth on $G_\alpha=G/K_\alpha$. We have 
\[
|f(xy)-f(x)|\le \int_{0}^{d_\alpha(e_\alpha,\pi_\alpha(y))}\Gamma_\alpha (\phi_\alpha,\phi_\alpha)(\pi_\alpha(x)\gamma^\alpha_y(s))^{\frac12}ds,
\]
where $\gamma^\alpha_y:[0,d_\alpha(e_\alpha,\pi_\alpha(y)) ] \to G_\alpha$ is a path from $e_{\alpha}$ to $\pi_{\alpha}(y)$ such that
\[
\left|\frac{d}{ds}\phi_{\alpha}(\pi_\alpha(x)\gamma^\alpha_y(s))\right|
\le \Gamma_\alpha (\phi_\alpha,\phi_\alpha)(\pi_\alpha(x)\gamma^\alpha_y(s))^{\frac12}.
\]
Indeed, one has
\begin{equation}\label{eq:grad}
\frac{d}{ds}\phi_{\alpha}(\pi_\alpha(x)\gamma^\alpha_y(s))=\sum_{i} a_i X_{\alpha,i}\phi_\alpha(\pi_\alpha(x)\gamma^\alpha_y(s)),
\end{equation}
where $a_i$ is a sequence of normalized coefficients so $\sum_i a_i^2<1$.

To see this, we introduce $X_{\alpha,i}=d\pi_\alpha(X_i)$ (only finitely many of those are non-zero) which form a H\"ormander system on $G_\alpha$ with associated distance $d_\alpha$. Observe that  
\[
\Gamma(f,f)=\sum_i |X_{\alpha,i}\phi_\alpha|^2\circ \pi_\alpha=\Gamma_\alpha(\phi_\alpha,\phi_\alpha)\circ\pi_\alpha.
\]
Then using the existence of the path $\gamma_{y}^{\alpha}\in G_{\alpha}$ and the classical inequality on $G_\alpha$, we have
\begin{equation}\label{eq:diff-bound}
|\phi_\alpha (\pi_\alpha(x)\pi_\alpha(y))-\phi_\alpha(\pi_\alpha(x))|\le \int_{0}^{d_\alpha(e_\alpha,\pi_\alpha(y))}\left|\sum_{i}|X_{\alpha,i}\phi_\alpha|^2(\pi_\alpha(x)\gamma^\alpha_y(s))\right|^{\frac12}ds.    
\end{equation}
 
Now integrating in $x$ over $G$, we get from the right-invariance of $\nu$ that 
\[
\int_G |f(xy)-f(x)|^pd\nu(x) \le  d_\alpha(e_\alpha ,\pi_\alpha (y))^p\int_G\Gamma (f,f)^{\frac p2}(z) d\nu(z) .
\]
Recall from \eqref{eq:dist-sup} that 
\[
d(e,y)= \sup_{\alpha} d_\alpha(e_\alpha,\pi_\alpha(y)).
\]  
 So, we get for any $f\in \mathcal B(G)$
\[
\int_G |f(xy)-f(x)|^pd\nu(x) \le  d(e ,y)^p\int_G\Gamma (f,f)^{\frac p2}(z) d\nu(z) .
\]
For the case $p=\infty$, the estimate can also be obtained in a similar way.     
\end{proof}
\begin{lem}\label{lem:GammaHt}
For any $f\in \mathcal B(G)$, we have
    \[
    \Gamma(H_tf, H_t f)^{\frac12}\le \int_t^{\infty}  \Gamma(\partial_s H_sf,\partial_s H_s f)^{\frac12}ds.
    \]
\end{lem}
\begin{proof}
We first claim that 
\[
\Gamma (H_tf, H_tf)^{\frac12}
=-\int_t^{\infty} \frac{\partial}{\partial s}\Gamma (H_sf, H_sf)^{\frac12} ds.
\]
Indeed, by the $L^2$ analyticity of $H_t$, one has $LH_{\infty}f=0$ and hence $\Gamma (H_{\infty}f, H_{\infty}f)=0$. Then the standard $\Gamma$-calculus gives 
\begin{align*}
\Gamma (H_tf, H_tf)^{\frac12}
=-\int_t^{\infty} \Gamma (H_sf, H_sf)^{-\frac12} \Gamma(H_sf, \partial_sH_sf) ds.
\end{align*}
By the Cauchy-Schwarz inequality, we obtain
\[
|\Gamma (H_sf, \partial_sH_sf)|\le \Gamma (H_sf, H_sf)^{\frac12}\Gamma (\partial_s H_sf, \partial_s H_sf)^{\frac12}.
\]   
The conclusion then follows.
\end{proof}

\begin{thm}\label{thm:LipDist<LipSG-Lp}
 Let $1< p<\infty$ and $0<\theta<1$. If $\Lambda_{\theta}^p(f)<\infty$, then $\mathrm L_{\theta}^p(f)<\infty$.   
\end{thm}

\begin{proof}
Assume that $\Lambda_{\theta}^p(f)<\infty$, we write
\[
f(xy)-f(x)=f(xy)-H_tf(xy)+H_tf(xy)-H_tf(x)+H_tf(x)-f(x).
\]

Let us start with the estimate of $\|H_tf-f\|_p$. Observe that 
\[
H_tf(x)-f(x) =\int_0^t \frac{\partial}{\partial s} H_sf(x)ds,
\]
then 
\begin{equation}\label{eq:diff-Htf-f}
    \|H_tf-f\|_p \le \int_0^t \left\|\partial_s H_sf\right\|_pds \le t^{\theta/2} \Lambda_{\theta}^p(f).
\end{equation}

Next, in view of Lemma \ref{lem:p-Poincare} and Lemma \ref{lem:GammaHt}, we have  
\begin{align*}
 \|\Delta_y H_tf\|_p 
 &\le d(e,y) \left\|\Gamma (H_tf, H_tf)^{\frac12}\right\|_p
\le d(e,y) \left\|\int_t^{\infty}\Gamma (\partial_sH_sf, \partial_sH_sf)^{\frac12}ds\right\|_p
\\ &\le
d(e,y) \int_t^{\infty}\left\|\Gamma (\partial_sH_sf, \partial_sH_sf)^{\frac12}\right\|_pds,
\end{align*}
where we have used the Minkowski inequality in the last line.

Now we apply the $L^p$ boundedness of the vector Riesz transform (see Theorem \ref{thm:LpRT}) and the $L^p$ analyticity of the heat semigroup, then
\begin{align*}
  \left\|\Gamma (\partial_s H_sf, \partial_s H_sf)^{\frac12} \right\|_p 
    & \le 
    2(p^*-1) \left\|\sqrt{L} \partial_s H_sf \right\|_p
    \\ & \le 
    4(p^*-1)\left\|\sqrt{L} H_{s/2}\partial_s H_{s/2}f \right\|_p 
    \\ & \le 
    4K\sqrt{p^*}(p^*-1)s^{-1/2}\left\|\partial_s H_{s/2}f \right\|_p,
\end{align*}
where $K$ is as in the proof of Corollary \ref{prop:gradient-Lp}. 

Combining the above estimates, we obtain that for $0<\theta<1$
\begin{equation}\label{eq:diff-Htf-xy}
    \left\|\Delta_y H_tf\right\|_p
    \le Cd(e,y) \int_t^{\infty} s^{-1/2}\left\|\partial_s H_{s/2}f \right\|_p ds
    \le Cd(e,y)t^{\frac{\theta-1}2} \Lambda_{\theta}^p(f).
\end{equation}
This estimate, together with \eqref{eq:diff-Htf-f}, implies that for any $y\in G$ and $1<p<\infty$
\begin{align*}
    \frac{\|\Delta_yf\|_p}{d(e,y)^{\theta}} 
    &\le \frac{2\|H_tf-f\|_p+\|\Delta_yH_tf\|_p}{d(e,y)^{\theta}} 
    \\ &\le 
    C\left(\frac{t^{\frac\theta2}}{d(e,y)^{\theta}} +\frac{t^{\frac{\theta-1}2}}{d(e,y)^{\theta-1}} \right)\Lambda_{\theta}^p(f).
\end{align*}
Letting $t=d(e,y)^{\frac12}$ yields that $\frac{\|\Delta_yf\|_p}{d(e,y)^{\theta}} \le C\Lambda_{\theta}^p(f)$ and we conclude the proof.
\end{proof}

\subsection{Comparison of Lipschitz seminorms: the case \texorpdfstring{$p=1,\infty$}{1infty}}
This section focuses on the comparison of Lipschitz norms for $p=1,\infty$. We have seen that Theorem \ref{thm:LipSG<LipDist} holds for $1\le p\le \infty$. It remains to prove that  $\mathrm L_{\beta}^p(f)\le C \Lambda_{\theta}^p(f)$ for $p=1, \infty$, and proper choices of $\theta$ and $\beta$. 

\begin{thm}\label{thm:LipDist<LipSG-L1}
Assume that $(\mu_t)_{t>0}$ satisfies $(\mathrm{CK}\lambda)$. Let $0<\theta<1$ and let $p=1$ or $\infty$. If $\Lambda_{\theta}^p(f)<\infty$, then $\mathrm L_{\beta}^p(f)<\infty$ for $\beta=\frac{\theta}{1+3\lambda}$.  
\end{thm}
\begin{proof}
Assume that $\Lambda_{\theta}^1(f)<\infty$. We use the same strategy as in Theorem \ref{thm:LipDist<LipSG-Lp} by writing
\begin{equation}\label{eq:difference}
f(xy)-f(x)=f(xy)-H_tf(xy)+H_tf(xy)-H_tf(x)+H_tf(x)-f(x).
\end{equation}
Note first that 
\begin{equation}\label{eq:diff-Htf-f-L1}
    \|H_tf-f\|_1 \le \int_0^t \left\|\frac{\partial}{\partial s} H_sf\right\|_1ds \le t^{\theta/2} \Lambda_{\theta}^1(f).
\end{equation}
Similarly as equation \eqref{eq:diff-Htf-xy}, one obtains from Lemma \ref{lem:p-Poincare} and Lemma \ref{lem:GammaHt} that
\[
\left\|\Delta_y H_tf\right\|_1 \le Cd(e,y)\int_t^{\infty}\left\|\Big|D\Big(\frac{\partial}{\partial s} H_sf\Big)\Big|_L \right\|_1 ds.
\]
Applying Theorem \ref{thm:L1-Linfty-gradient}, then
\[
\left\|\Big|D\Big(\frac{\partial}{\partial s} H_sf\Big)\Big|_L \right\|_1
=\frac12 \left\|\Big|DH_{s/2}\Big(\frac{\partial}{\partial s} H_{s/2}f\Big)\Big|_L \right\|_1
\le C s^{-\frac12}\max\{s^{-\frac32\lambda}, 1\}\left\|\frac{\partial}{\partial s} H_{s/2}f\right\|_1.
\]
We thus have for $0<t<1$ 
\[
\left\|\Delta_y H_tf\right\|_1 \le C d(e,y) \int_t^{\infty} s^{-\frac{1-\theta}2}\max\{s^{-\frac32\lambda}, 1\} \frac{ds}{s} \Lambda_{\theta}^1(f)
\le C d(e,y) t^{\frac{\theta-1-3\lambda}{2}} \Lambda_{\theta}^1(f).
\]
This estimate, together with \eqref{eq:diff-Htf-f-L1}, gives  that for any $y\in G$,
\begin{align*}
    \frac{\|\Delta_yf\|_1}{d(e,y)^{\beta}} 
    &\le \frac{2\|H_tf-f\|_1+\|\Delta_yH_tf\|_1}{d(e,y)^{\beta}} 
    \\ &\le 
    C\left(\frac{t^{\frac\theta2}}{d(e,y)^{\beta}} +\frac{t^{\frac{\theta-1-3\lambda}{2}}}{d(e,y)^{\beta-1}} \right)\Lambda_{\theta}^1(f).
\end{align*}
Let $\beta=\frac{\theta}{1+3\lambda}$ and choose $t$ to be $d(e,y)^{\frac{2\beta}{\theta}}$, then we obtain $\mathrm L_{\beta}^1(f)<C\Lambda_{\theta}^1(f)$ by taking the supremum for $y\in G$.

For $p=\infty$, we still write the difference $f(xy)-f(x)$ as in \eqref{eq:difference}. It easily follows that 
\[
\|H_tf-f\|_{\infty} \le t^{\theta/2} \Lambda_{\theta}^{\infty}(f).
\]
On the other hand, we deduce from Lemma \ref{lem:p-Poincare}, Lemma \ref{lem:GammaHt} and  Theorem \ref{thm:L1-Linfty-gradient} that
\begin{align*}
\left\|\Delta_y H_tf\right\|_{\infty} \le C d(e,y) \int_t^{\infty}\left\|\Big|D\Big(\frac{\partial}{\partial s} H_sf\Big)\Big|_L \right\|_{\infty} ds\le C d(e,y) t^{\frac{\theta-1-3\lambda}{2}} \Lambda_{\theta}^{\infty}(f).
\end{align*}
The rest of the proof is then the same as the case $p=1$ and we omit the details.

\end{proof}


\section{Regularity of solutions }

We now consider the regularity problem stated in the introduction. Namely, given a (bi-invariant) Laplacian, $L$, on our compact group $G$ and a ``weak solution'', $U$, of $LU=F$ in an open subset $\Omega$ in which $F$ is H\"older continuous, can we say something about the regularity of $U$? (e.g., is it continuous?). In order to make this precise, we need a good notion of what ``weak solution'' means and it is also clear from previous sections that we have at least two possible choices
for what being H\"older continuous means.
In the next subsection, we follow previous work by the first and third authors introducing proper spaces of distributions which provide a clear and useful notion of ``distribution solution'' for $LU=F$. Then we proceed in studying the regularity problem, first in the context of the semigroup Lipschitz spaces $\Lambda^p_\theta$ and then in the context of the distance Lipschitz spaces $\mathrm L^p_\theta$.

\subsection{Spaces of test functions and distributions \texorpdfstring{$\mathcal T_L, \mathcal T_L^{\,'}$}{TL}}\label{sec:distribution}
We will use a space of test function $\mathcal T_L$ and its distribution space $\mathcal T_L^{\,'}$ introduced in \cite{BSC2005, BSC2006}.

Recall that we fix a projective family $X=(X_i)_{i\in I}$ of $\mathfrak G$ and $L=-\sum_I X_i^2$. Fix $k,n\in \mathbb N=\{0,1,2,\cdots\}$ and consider the set 
\[
\Lambda(k,n)=\left\{\lambda=(\lambda_0, \lambda_1, \cdots, \lambda_k); \lambda_i\in \mathbb N, \,\sum \lambda_i=n\right\}.
\]
For $f\in \mathcal B(G)$, $\ell \in I^k$, $\lambda\in \Lambda(k,n)$, set
\begin{equation}\label{eq:P-ell-lambda}
P^{\ell, \lambda}f=L^{\lambda_0}X_{\ell_1}L^{\lambda_1}X_{\ell_2}\cdots L^{\lambda_{k-1}}X_{\ell_k} L^{\lambda_{k}} f.   
\end{equation}
\begin{defn}\label{def:M_Lk}
For each $k\in \mathbb N$, consider the following seminorm on $\mathcal B(G)$:
\[
M_L^k(f)=\sup_{\substack{m,n\in \mathbb N\\n+2m\le k}} \sup_{\lambda\in \Lambda(n,m)}\left\| \left(\sum_{\ell\in I^n}|P^{\ell,\lambda}f|^2\right)^{1/2}\right\|_{\infty}.
\]
We define $\mathcal T_L^k$ as the completion of $\mathcal B(G)$ with respect to the seminorm $M_L^k$. Moreover, we set $\mathcal T_L=\bigcap_{k\in \mathbb N}\mathcal T_L^k$ and equip it with its natural system of seminorms and the
corresponding topology.

\end{defn}
\begin{remark}
For an open set $\Omega\subset G$, we set 
\begin{equation}\label{eq:local-Mk}
M_L^k(\Omega, f)=\sup_{\substack{m,n\in \mathbb N\\n+2m\le k}} \sup_{\lambda\in \Lambda(n,m)}\sup_{x\in \Omega} \left(\sum_{\ell\in I^n}|P^{\ell,\lambda}f(x)|^2\right)^{1/2}.    
\end{equation}
\end{remark}

The space $\mathcal T_L$ is Fr\'echet space. Denote by $\mathcal T_L^{\,'}$ the strong topological dual of $\mathcal T_L$. That is, any element $U\in \mathcal T_L^{\,'}$ is a linear functional on $\mathcal T$ such that there exist $m=m(U)\in \mathbb N$ and $C=C(U)>0$ such that
\[
|U(\phi)|\le C M_L^m(\phi), \quad \forall \phi\in \mathcal T_L.
\]
The topology of $\mathcal T_L^{\,'}$ is defined by the family of seminorms 
\[
p_B(U)=\sup_{\phi\in B}|U(\phi)|,
\]
where $B$ runs over all bounded sets in $\mathcal T_L$ (i.e., $\sup_{\phi\in B}M_L^k(\phi)<\infty$ for all $k\in \mathbb N$).

Now we introduce some notation taken from \cite{BSC2006} which will be used later. 
The convolution of two functions $u,v\in \mathcal B(G)$ is defined by 
\[
u*v(x)=\int_G u(y)v(y^{-1}x)d\nu(y)=\int_G u(xy^{-1})v(y)d\nu(y).
\]
For $f\in \mathcal B(G)$, the left and right convolution by a measure $\mu$ is defined by
\[
\mu*f(x)=\int_Gf(y^{-1}x)d\mu(y),\quad f*\mu(x)=\int_Gf(xy^{-1})d\mu(y).
\]
Convolutions of $U\in \mathcal B'(G)$ and $f\in \mathcal B(G)$ are given by 
\[
[U*f](\phi)=U(\check{f}*\phi), \,\,[f*U](\phi)=U(\phi*\check{f}), \quad \forall \phi\in \mathcal B(G),
\]
where $\check{f}(x)\coloneqq f(x^{-1})$. Moreover, $U*f$ and $f*U$ are functions in $\mathcal B(G)$ and
\[
U*f(x)=U(\mathcal L_{x^{-1}}\check{f}), \,\,f*U(\phi)=U(\mathcal R_{x^{-1}}\check{f}),
\]
where $\mathcal L_{a}$ and $\mathcal R_{a}$ ($a\in G$) are the left and right translation operators defined by 
\[
\mathcal L_{a}f(x)=f(ax), \quad\mathcal R_{a}f(x)=f(xa).
\]

Given a Borel measure $\mu$, define $\check{\mu}$ by $\check{\mu}(A)=\mu(A^{-1})$ for any Borel set $A$. Under the assumption that $(\mu_t)_{t>0}$ satisfies $(\mathrm{CK})$ (see Definition \ref{def:CK etc}), Lemma 3.1 in \cite{BSC2006} shows that the density $x\mapsto \check{\mu}_t(x)$ with respect to $\nu$ belongs to $\mathcal T_L$. Finally, we finish the subsection with the following useful result. 
\begin{prop}[\cite{BSC2005}, \protect{\cite[Proposition 2.7]{BSC2006}}]\label{prop:convolution}
Assume that $L$ is bi-invariant.
\begin{enumerate}
    \item Let $U\in \mathcal T_L^{\,'}$ and $f\in  \mathcal T_L$. Then $U*f\subset  \mathcal T_L$ and 
        \[
        X^{\ell} L^n[U*f]=U*[X^{\ell} L^nf].
        \]
    \item Let $\phi_{\epsilon}, \epsilon>0$ be a family of functions in $\mathcal T_L$ such that $\phi_{\epsilon}\to \delta_e$ as $\epsilon$ tends to zero. Then for any $U\in \mathcal T_L^{\,'}$, the distribution $U^{\epsilon}=U*\phi_{\epsilon}$ is represented by a function in $\mathcal T_L$, and 
        \[
        \lim_{\epsilon\to 0}U^{\epsilon}(\psi)=U(\psi), \quad\forall\psi\in  \mathcal T_L.
        \]
\end{enumerate}
\end{prop}


\subsection{Regularity for $\mathcal L_p$ solutions}
In this subsection, we will prove a priori regularity results of the Poisson equation $Lu=f$. 
We first introduce a Sobolev-type space $\mathcal L_p$, $1\le p\le \infty$, where the solution is assumed to sit in. 
\begin{defn}
 For $1\le p\le\infty$, we define $\mathcal L_p(G)$ as the completion of the set 
\[
\{u\in\mathcal B(G);\|u\|_p+\|Lu\|_p<\infty\}
\]
with respect to the norm $\|\cdot\|_{\mathcal L_p}$ given by 
$\|u\|_{\mathcal L_p}\coloneqq\|u\|_p+\|Lu\|_p$.   
\end{defn}
We point out that 
\begin{enumerate}[(a)]
\item For $1<p<\infty$, $\mathcal L_p(G)$ is the domain of the $L^p$-generator of $H_t$. 
 
\item For $p=\infty$, $\mathcal L_{\infty}(G)$ is 
the set of continuous functions $u$ such that $Lu$ is also continuous.
\end{enumerate}

In what follows, denote 
\begin{equation}
 f=Lu, \quad u\in \mathcal L_p(G)  
\end{equation}
so $f\in L^p(G)$ and $\int_Gfd\nu=0$. In this case, we say that  $f\in L_0^p(G)$, where 
\[
L_0^p(G)\coloneqq\left\{f\in L^p(G);\int_Gfd\nu=0\right\}.
\]

Now we state the following a priori regularity result for the Poisson equation. 

\begin{thm}\label{thm:StrongSol-Poisson}
For $1<p<\infty$, let $u\in \mathcal L_p(G)$.
Then  $ X_i u, X_iX_ju\in L^p(G)$. Moreover,  $-\sum_{i\in I}X_i^2 u$ converges strongly in $L^p(G)$.
\end{thm}

The proof relies on the following lemma, which is also useful in the sequel. 

\begin{lem}\label{lem:fracLaplacian-Lp}
 Let $\delta>0$ and let $1<p<\infty$.

\begin{enumerate}
\item For $f\in L^p(G)$, there holds $\|(I+L)^{-\delta}f\|_p\le \|f\|_p$. In particular,
\[
\|\sqrt L(I+L)^{-1}f\|_p \le C\sqrt{p^*}\|f\|_p,
\]
where the constant $C$ is independent of $p$. 
\item Assume that $L$ has a positive $L^2$-spectral gap $\lambda_1$. Then for $f\in L_0^p(G)$
\[
\|L^{-\delta}f\|_p\le \left(\frac{p^*}{2\lambda_1}\right)^{\delta}\|f\|_p,
\]
where 
$p^*=\max\{p,\frac p{p-1}\}$. 
\end{enumerate}
\end{lem}
\begin{remark}\label{rem:LpL-1}
As a simple consequence of (i), we also have 
\[
\|L(I+L)^{-1}f\|_p\le \|f\|_p+\|(I+L)^{-1}f\|_p\le2\|f\|_p.
\]
Besides, it follows from Remark \ref{rem:spectral gap} that the property $(\mathrm{CK})$ implies the assumption of positive $L^2$-spectral gap in (ii).    
\end{remark}
\begin{proof}
We first prove (i). Note that 
$(I+L)^{-\delta}f=\frac1{\Gamma(\delta)} \int_0^{\infty} t^{\delta-1}e^{-t}H_tfdt$. It follows from the $L^p$ contractivity of the heat semigroup that 
\[
\|(I+L)^{-\delta}f\|_p\le \frac1{\Gamma(\delta)} \int_0^{\infty} t^{\delta-1}e^{-t}\|H_tf\|_pdt\le \|f\|_p.
\]
In the same manner, one gets
\[
\|\sqrt L(I+L)^{-1}f\|_p \le \int_0^{\infty} e^{-t}\|\sqrt LH_tf\|_p dt \le \sqrt{\pi}K\sqrt{p^*}\|f\|_p,
\]
where we have used the multiplicative inequality \eqref{eq:multip-ineq} and the $L^p$ analyticity of $H_t$ in the second inequality.

Next turning to the proof of (ii). From the assumption, for any  $f\in L_0^2(G)$ we have
\begin{equation}\label{eq:HS-L2}
\|H_tf\|_2\le e^{-\lambda_1 t}\|f\|_2. 
\end{equation}
On the other hand, the $L^1$ contractivity of $H_t$ gives that $\|H_tf\|_1\le \|f\|_1$. Thus applying the Riesz-Thorin interpolation theorem, we have for $1<p<2$ 
\[
\|H_tf\|_p \le e^{-2\lambda_1(1-\frac1p)t}\|f\|_p.
\]

Now write $L^{-\delta}f=\frac1{\Gamma(\delta)} \int_0^{\infty} t^{\delta-1}H_tfdt$. Then for $1<p<2$
\begin{align*}
\|L^{-\delta}f\|_p
&\le \frac1{\Gamma(\delta)}\int_0^{\infty} t^{\delta-1}\|H_tf\|_pdt
\le \frac1{\Gamma(\delta)}\int_0^{\infty} t^{\delta-1} e^{-2\lambda_1(1-\frac1p)t}dt\, \|f\|_p
\\ &= (2\lambda_1)^{-\delta}\left(\frac{p}{p-1}\right)^{\delta} \|f\|_p.
\end{align*}

Similarly, in view of \eqref{eq:HS-L2} and the $L^{\infty}$ contractivity of $H_t$, then the Riesz-Thorin interpolation theorem gives that for $2<p<\infty$
\[
\|H_tf\|_p \le e^{-\frac{2\lambda_1}{p}t}\|f\|_p
\] 
and therefore
\[
\|L^{-\delta}f\|_p
\le \frac1{\Gamma(\delta)}\int_0^{\infty} t^{\delta-1} e^{-\frac{2\lambda_1}{p}t}dt\, \|f\|_p
=  (2\lambda_1)^{-\delta} p^{\delta} \|f\|_p.
\]
We thus conclude the proof by summarizing the two cases.
\end{proof}

\begin{proof}[Proof of Theorem \ref{thm:StrongSol-Poisson}]
For $u\in \mathcal L_p(G)$, we write $f=Lu$. Hence $f\in L^p(G)$ and  $u=(I+L)^{-1}(u+f)$. 
In order to show $X_iu\in L^p(G)$, $i\in I$, we apply the $L^p$ boundedness of Riesz transforms in Theorem \ref{thm:LpRT} as well as Lemma \ref{lem:fracLaplacian-Lp} (i).
Then 
\begin{align*}
\|X_iu\|_p&=\|X_i(I+L)^{-1}(u+f)\|_p=\left\|X_iL^{-\frac12}L^{\frac12}(I+L)^{-1}(u+f)\right\|_p
\\ &\le \left\|X_iL^{-\frac12}\right\|_{p\to p}\left\|L^{\frac12}(I+L)^{-1}(u+f)\right\|_{p}
\le  C(\|u\|_p+\|f\|_p).
\end{align*}
Furthermore, it follows from the $L^p$ boundedness of second order Riesz transforms in Theorem \ref{thm:Lp2ndRT} that $X_iX_ju\in L^p(G)$. Indeed, for any $i, j\in I$
\[
\|X_iX_ju\|_p=\|X_iX_jL^{-1}f\|_p\le C\|f\|_p.
\]

Finally, we will show the strong convergence in $L^p(G)$ by duality, similarly as in \cite{Bendikov1995}. 
Recall that by Theorem \ref{thm:LpRT} we have for any $\varphi\in \mathcal B(G)$, 
\begin{equation}\label{eq:Rmn-bound}
\left\|R^G\varphi\right\|_p=\Bigg\|\Bigg(\sum_{i\in I}|R_i\varphi|^2\Bigg)^{1/2}\Bigg\|_p \le 2(p^*-1)\|\varphi\|_p.
\end{equation}
Now let $\varphi\in \mathcal B(G)$ and recall that for $m<n$, $\overrightarrow{R}_{mn}=(0,\cdots, R_m, R_{m+1},\cdots, R_n, 0, \cdots)$.
Denoting by $q$ the conjugate of $p$, it follows from the Cauchy-Schwarz  inequality and H\"older's inequality that 
\begin{align*}
    \int_G \phi(x)\sum_{i=m}^n X_i^2 u(x)d\nu(x)
    &=\int_G \phi(x)\sum_{i=m}^n R_i^2 f(x)d\nu(x)
    = \int_G \sum_{i=m}^nR_i\phi(x)  R_i f(x)d\nu(x)
    \\ & \le \left\|\overrightarrow{R}_{mn}\phi\right\|_q \left\|\overrightarrow{R}_{mn} f\right\|_p 
    \le 2(p^*-1)\|\phi\|_q \left\|\overrightarrow{R}_{mn} f\right\|_p,
\end{align*}
where we have used \eqref{eq:Rmn-bound} in the last inequality (note that $p^*=q^*$ since $p,q$ are conjugates).
Hence by duality, we obtain
\begin{equation}\label{eq:convergenceLp}
\left\|\sum_{i=m}^n X_i^2 u\right\|_p\le 2(p^*-1) \left\|\overrightarrow{R}_{mn} f\right\|_p.
\end{equation}
In view of Remark \ref{rem:Rmn-to-0} (ii), letting $m, n\to \infty$ deduces that $\left\|\sum_{i=m}^n X_i^2 u\right\|_p $ converges to zero and we conclude the proof.   
\end{proof}

\paragraph{A priori Lipschitz regularity}
\begin{thm}\label{thm:Lip-Regularity-Lp}
Let $\theta>0$ and $1<p<\infty$. Let $u\in \mathcal L_p(G)$ and $f=Lu$. Assume further $f\in\Lambda_{\theta}^p(G)$.
Then we have $\Lambda_{\theta+2}^p(u),\Lambda_{\theta+1}^p(X_iu), \Lambda_{\theta}^p(X_iX_ju)<\infty$.  Moreover,  $-\sum_{i\in I}X_i^2 u$ converges to $Lu$ with respect to  $\Lambda_{\theta}^p(\cdot)$.
\end{thm}

\begin{proof}
For $1<p<\infty$ and $\theta>0$,  suppose that $\Lambda_{\theta}^p(f)<\infty$, i.e.,
\[
\Lambda_{\theta}^p(f)=\sup_{t>0}  t^{n-\frac\theta2} \left\|L^nH_tf\right\|_p<\infty,
\]
where $n$ is the largest integer greater than $\frac\theta2$. 

In order to show that $\Lambda_{\theta}^p(u)<\infty$, we observe that $n+1>\frac{\theta+2}{2}$. Then by Lemma \ref{lem:equiv-Lip-sg},
\begin{align*}
    \Lambda_{\theta+2}^p(u)
    =\sup_{t>0} t^{n-\frac\theta2} \left\|L^{-1} L^{n+1} H_tf\right\|_p
    \le C \sup_{t>0} t^{n-\frac\theta2} \left\|L^n  H_tf\right\|_p<\infty.
\end{align*}

Similarly, since $n+\frac12>\frac{\theta+1}{2}$, we have that
\[
\Lambda_{\theta+1}^p(X_iu)
=\sup_{t>0} t^{n+\frac12-\frac{\theta+1}{2}} \left\|L^{n+\frac12} H_tX_iu\right\|_p=\sup_{t>0} t^{n-\frac{\theta}{2}} \left\|L^{n} H_tR_iL^{-\frac12}f\right\|_p.
\]
Applying Propsotion \ref{prop:RTinLip} (i) and (ii),
we then obtain 
\[
    \Lambda_{\theta+1}^p(X_i u)
    =\Lambda_{\theta+1}^p(R_iL^{-\frac12} f)
    \le 2(p^*-1)\Lambda_{\theta+1}^p(L^{-\frac12} f)
    = 2(p^*-1)\Lambda_{\theta}^p(f)<\infty,
\]
and
\begin{align*}
    \Lambda_{\theta}^p(X_iX_ju)
    &=\Lambda_{\theta}^p(R_iR_jf)
    \le  2(p^*-1)\Lambda_{\theta}^p(f)<\infty.
\end{align*}

Finally, we claim that $-\sum_{i\in I}X_i^2 u$ converges to $Lu$ in the Lipschitz norm $\Lambda_{\theta}^p$. It suffices to show that for any $n>m$, one has $\Lambda_{\theta}^p\left(\sum_{i=m}^n X_i^2u\right) \to 0$ as $m,n\to \infty$. Indeed, on account of  \eqref{eq:convergenceLp} and Proposition \ref{prop:RTinLip} (iii), there holds
\begin{align*}
 \Lambda_{\theta}^p\left(\sum_{i=m}^n X_i^2u\right)
 &=\sup_{t>0} t^{n-\frac\theta2} \left\|\sum_{i=m}^n X_i^2 L^n H_tu\right\|_p
 \\ &\le 2(p^*-1)\sup_{t>0} t^{n-\frac\theta2} \left\|\overrightarrow{R}_{mn} L^n H_tf\right\|_p \longrightarrow 0, \quad \text{as } m,n\to \infty.
\end{align*}
\end{proof}


\begin{cor}\label{cor:Lip-Regularity-dist-Lp}
Let $0<\theta<1$ and $1<p<\infty$. Assume that $(\mu_t)_{t>0}$ satisfies $(\mathrm{CK}\lambda)$ with $0<\lambda<\frac{\theta}{\theta+2}$. Let $u\in \mathcal L_p(G)$ and $f=Lu$. Assume further $\mathrm L_{\theta}^p(f)<\infty$. 
Then we have $\mathrm L_{\beta}^p(u),\mathrm L_{\beta}^p(X_iu), \mathrm L_{\beta}^p(X_iX_ju)<\infty$ for any $0<\beta\le (1-\lambda)\theta-2\lambda$.  Moreover,  $-\sum_{i\in I}X_i^2 u$ converges to $Lu$ with respect to  $\mathrm L_{\beta}^p(\cdot)$. In particular, if $(\mu_t)_{t>0}$ satisfies $(\mathrm{CK}0^+)$, then the same conclusion holds for $0<\beta<\theta$. 
\end{cor}
\begin{proof} The proof can be briefly summarized in the diagram below:
\[
\begin{tikzcd}[cramped,column sep=huge]
 \mathrm L_{\theta}^{p}(f)<\infty \arrow[r,"\text{Thm. \ref{thm:LipSG<LipDist}}"] & \Lambda_{\beta}^{p}(f)<\infty\quad \arrow[r,"\text{Theorem \ref{thm:Lip-Regularity-Lp}}", "\text{Rem. \ref{rem:Lip-sg-p-large} (iii), Lem. \ref{lem:equiv-Lip-sg}}"']&\quad   \text{Regularity and convergence in }\Lambda_{\beta}^p(\cdot) \arrow[d, "\text{Theorem \ref{thm:LipDist<LipSG-Lp}}"] \\
            &             & \text{ \,\,Regularity and convergence in }\mathrm L_{\beta}^p(\cdot)         
\end{tikzcd}
\]
Indeed, for $1<p<\infty$ we note that from Theorem \ref{thm:LipSG<LipDist}, there holds $\Lambda_{\beta}^p(f)<\infty$ for $0<\beta\le (1-\lambda)\theta-2\lambda$. It then follows from Theorem \ref{thm:Lip-Regularity-Lp} that 
\[
\Lambda_{\beta+2}^p(u),\Lambda_{\beta+1}^p(X_iu), \Lambda_{\beta}^p(X_iX_ju)<\infty,
\]
and $-\sum_{i\in I}X_i^2 u$ converges to $Lu$ with respect to  $\Lambda_{\beta}^p(\cdot)$.
Thanks to Remark \ref{rem:Lip-sg-p-large} (iii) and Lemma \ref{lem:equiv-Lip-sg}, one also has $\Lambda_{\beta}^p(u),\Lambda_{\beta}^p(X_iu),\Lambda_{\beta}^p(X_iX_ju)<\infty$. Hence, in view of Theorem \ref{thm:LipDist<LipSG-Lp}, we conclude that $\mathrm L_{\beta}^p(u), \mathrm L_{\beta}^p(X_iu), \mathrm L_{\beta}^p(X_iX_ju)<\infty$. 
\end{proof}

We finish this subsection by working on the case $p=1$ or $\infty$.
\begin{thm}\label{thm:Lip-Regularity-Linfty}
Let $0<\theta<1$ and let $p=1$ or $\infty$. Assume that $(\mu_t)_{t>0}$ satisfies $(\mathrm{CK}\lambda)$ with $0<\lambda<\frac\theta2$. For any $u\in \mathcal L_\infty(G)$ and $f=Lu$, assume in addition $\Lambda_{\theta,2}^{p}(f)<\infty$.
Then for any $i,j \in I$, we have
\[\Lambda_{\theta+2,3}^{p}(u),\Lambda_{\theta+1-3\lambda,2}^{p}(X_iu),\Lambda_{\theta-2\lambda,1}^{p}(X_iX_ju)<\infty.
\]
Moreover,  $-\sum_{i\in I}X_i^2 u$ converges to $Lu$ with respect to  $\Lambda_{\theta-2\lambda,1}^{p}(\cdot)$.
\end{thm}

\begin{proof}
We start with the proof for $p=\infty$.
Assume that $\Lambda_{\theta,2}^{\infty}(f)<\infty$. 
It follows from Lemma \ref{lem:Lip-sg-p-alter-1infty}  that $\Lambda_{\theta,1}^{\infty}(f)<\infty$. We first note that as a direct consequence of Proposition \ref{prop:RTinLip-infty} (ii),  there holds that $\Lambda_{\theta-2\lambda,1}^{\infty}(X_iX_ju)<\infty$. 

Next, we show $\Lambda_{\theta+2,3}^{\infty}(u)<\infty$ by applying \eqref{eq:ultra-Lp-Linfty}. Then for any $1< q<\infty$, one has
\begin{align*}
     \sup_{0<t<1}  t^{3-\frac{\theta+2}{2}}\left\|\frac{\partial^3}{\partial t^3} H_tu\right\|_{\infty}
    &
    \le  \sup_{0<t<1}  t^{3-\frac{\theta+2}{2}}\|H_{t/2}\|_{q\to \infty} \left\|L^3 H_{t/2} u\right\|_{q}
    \\&\le  \sup_{0<t<1}  t^{2-\frac{\theta}{2}}e^{M_0(t/2)/q} \left\|L^2 H_{t/2}f\right\|_{q}
    \\& \le C\sup_{0<t<1} t^{2-\frac\theta2}\left\|L^2 H_{t}f\right\|_{\infty}<\infty,
\end{align*}
where for the third inequality we pick $q>1$ as a function of $t\in(0,1)$ such that $M_0(t/2)/q\le1$. Since $\theta+2<3$, one obtains $\Lambda_{\theta+2,3}^{\infty}(u)<\infty$ from Remark \ref{rem:Lip-sg-p-large} (i).

Now we move  to the proof for  $\Lambda_{\theta+1-3\lambda,2}^{\infty}(X_iu)<\infty$. Proposition \ref{prop:RTinLip-infty} (i) gives that $\Lambda_{\theta+1-3\lambda,2}^{\infty}(X_iu)\le \Lambda_{\theta+1-\lambda,2}^{\infty}(L^{-\frac12}f)$. To estimate $\Lambda_{\theta+1-\lambda,2}^{\infty}(L^{-\frac12}f)$,
the method is similar as above. For any $1< q<\infty$, we use  again \eqref{eq:ultra-Lp-Linfty},  the multiplicative inequality \eqref{eq:multip-ineq} and the $L^q$ analyticity of $H_t$. Then
\begin{align*}
    \sup_{0<t<1} t^{2-\frac{\theta+1-\lambda}{2}}\left\|L^2 H_tL^{-\frac12}f\right\|_{\infty}
    & 
    \le  \sup_{0<t<1} t^{2-\frac{\theta+1-\lambda}{2}}e^{M_0(t/2)/q}  \left\|L^{\frac32} H_{t/2}f\right\|_{q}
    \\&\le C  \sup_{0<t<1} t^{2-\frac{\theta+1-\lambda}{2}}\sqrt{q^*} e^{M_0(t/2)/q} t^{-\frac12} \left\|L H_{t/2}f\right\|_{q}
    \\& \le C\sup_{0<t<1/2} t^{1-\frac\theta2}\left\|L H_{t}f\right\|_{\infty}<\infty, 
\end{align*} 
where for the third inequality we pick $q>2$ as a function of $t$ such that $M_0(t/2)/q \simeq 1$ and $\sqrt{q^*}\simeq t^{-\frac{\lambda}2}$.
By Remark \ref{rem:Lip-sg-p-large} (i) and Definition \ref{def:Lip-sg}, we conclude that $\Lambda_{\theta+1-2\lambda,2}^{\infty}(X_iu)<\infty$.

To finish the proof for $p=\infty$, it remains to justify the convergence of $-\sum_{i\in I}X_i^2 u$ to $Lu$ in Lipschitz norm $\Lambda_{\theta-2\lambda,1}^{\infty}$. This can be obtained in a similar manner as the proof of Theorem \ref{thm:Lip-Regularity-Lp}. More precisely, we pick $q>2$ as above so $M_0(t/2)/q \simeq 1$ , then for any $m,n\in \mathbb N$, $m>n$
\begin{align*}
\sup_{0<t<1} t^{1-\frac\theta2+\lambda} \left\|\frac{\partial}{\partial t} H_t\sum_{i=m}^nX_i^2 u\right\|_{\infty}
&\le 
\sup_{0<t<1} t^{1-\frac\theta2+\lambda} \|H_{t/2}\|_{q\to \infty} \left\|\sum_{i=m}^nX_i^2L H_{t/2} u\right\|_{q}
\\&\le2\sup_{0<t<1} (q^*-1) e^{M_0(t/2)/q}t^{1-\frac\theta2+\lambda}\left\|\overrightarrow{R}_{mn}L H_{t/2}f\right\|_{q}
\\ &\le C\sup_{0<t<1} t^{1-\frac\theta2}\left\|\overrightarrow{R}_{mn}L H_{t/2}f\right\|_{q}, 
\end{align*}
where the second inequality is owning to \eqref{eq:convergenceLp} and \eqref{eq:ultra-Lp-Linfty}. 
From Proposition \ref{prop:RTinLip} (iii), we have that  $\sup_{0<t<1} t^{1-\frac\theta2}\left\|\overrightarrow{R}_{mn}L H_{t/2}f\right\|_{q}\to 0$ as $m,n\to \infty$ and thus  $\Lambda_{\theta-2\lambda,1}^{\infty}(\sum_{i\in I}X_i^2 u)\to 0$ as $m,n\to \infty$. 

\medskip
When $p=1$, the argument is similar. Instead of applying \eqref{eq:ultra-Lp-Linfty}, we follow the idea in the proof of Theorem \ref{thm:L1-diff} (i) by using \eqref{eq:ultracontractive} (see also Theorem \ref{thm:L1-Linfty-gradient} and Lemma \ref{lem:Lip-sg-p-alter-1infty}). The detail is omitted here.
\end{proof}

\begin{cor}
Let $0<\theta<1$ and $p=1$ or $\infty$. Assume that $(\mu_t)_{t>0}$ satisfies $(\mathrm{CK}\lambda)$ such that $0<\lambda<\frac{\theta}{\theta+6}$. Let $u\in \mathcal L_p(G)$ and $f=Lu$. Assume further $\mathrm L_{\theta}^p(f)<\infty$. 
Then we have $\mathrm L_{\frac{\beta-4\lambda}{1+3\lambda}}^p(u),\mathrm L_{\frac{\beta-4\lambda}{1+3\lambda}}^p(X_iu), \mathrm L_{\frac{\beta-4\lambda}{1+3\lambda}}^p(X_iX_ju)<\infty$ for any $0<\beta\le (1-\lambda)\theta-2\lambda$.  Moreover,  $-\sum\limits_{i\in I}X_i^2 u$ converges to $Lu$ with respect to  $\mathrm L_{\frac{\beta-4\lambda}{1+3\lambda}}^p(\cdot)$.

\end{cor}
\begin{proof}
We only show the proof for $p=\infty$, which follows from the diagram below:
\[\begin{tikzcd}
  \mathrm L_{\theta}^{\infty}(f)<\infty  
  \arrow[r,"\text{Theorem \ref{thm:LipSG<LipDist}}"] & \Lambda_{\beta}^{\infty}(f)<\infty
 \arrow[r,"\text{Lemma \ref{lem:Lip-sg-p-alter-1infty}}"] & \Lambda_{\beta-2\lambda,2}^{\infty} (f)<\infty
\arrow[d,"\text{Theorem \ref{thm:Lip-Regularity-Linfty}}"] \\ -\sum\limits_{i}X_i^2 u \text{ converges in }\mathrm L_{\frac{\beta-4\lambda}{1+3\lambda}}^{\infty}(\cdot) & & -\sum\limits_{i}X_i^2 u \text{ converges in }\Lambda_{\beta-4\lambda}^{\infty}(\cdot)   \arrow[ll,"\text{Theorem \ref{thm:LipDist<LipSG-L1}}"] 
\end{tikzcd}
\]
Apart from Theorem \ref{thm:Lip-Regularity-Linfty}, the regularity of  $u, X_iu, X_iX_ju$ in $\Lambda_{\beta-4\lambda}^p(\cdot)$ is also due to  Remark \ref{rem:Lip-sg-p-large} (iii) and Lemma \ref{lem:Lip-sg-p-alter-1infty} (ii).
\end{proof}

\subsection{Regularity for global distribution solutions}
Now we interpret the Poisson equation in the distributional sense. 
Consider two distributions $U, F\in \mathcal T_L^{\,'}$ such that 
\begin{equation}\label{eq:Poisson-distribution}
LU=F \quad \text{ in } \mathcal T_L^{\,'}.
\end{equation}

We say that $F\in L_0^p(G)$ if the distribution $F$ can be represented by a function in $f\in L_0^p(G)$, that is, for all $\phi\in \mathcal B(G)$, $F(\phi)=\int_G f\phi d\nu$.

\begin{prop}\label{prop:distributionLp}
 Let $1<p<\infty$. Assume that $(\mu_t)_{t>0}$ satisfies the property $(\mathrm{CK})$. Consider the Poisson equation \eqref{eq:Poisson-distribution}. Assume that $F\in L_0^p(G)$, then we have $U, X_i U, X_iX_j U\in L^p(G)$, for any $i, j\in I$. 
Moreover, $-\sum_{i\in I}X_i^2U$ converges to $LU$ in $L^p(G)$. 
\end{prop}

\begin{proof}
Let $F,U\in \mathcal T_L^{\,'}$ satisfy the equation \eqref{eq:Poisson-distribution} and assume $F\in L_0^p(G)$. By Remark \ref{rem:spectral gap}, the property $(\mathrm{CK})$ implies that $L$ has positive $L^2$-spectral gap.
Then for any $\varphi\in \mathcal T_L\cap L^q(G)$, Lemma \ref{lem:fracLaplacian-Lp} (ii) gives that
\[
\langle U, \varphi\rangle=\langle L^{-1}F, \varphi\rangle
\le \|L^{-1}F\|_{p} \|\varphi\|_{q}
\le C \|F\|_{p} \|\varphi\|_{q},
\]
and hence $U\in L^p(G)$.

Next we turn to the proof that $X_iU\in L^p$, for $i\in I$. One also has $X_iU\in \mathcal T_L^{\,'}$ and the proof is analogue as above. Indeed, for any $\varphi\in \mathcal T_L\cap L^q(G)$, 
\begin{align*}
\langle X_iU, \varphi\rangle &=\langle L^{-1/2}F, X_iL^{-1/2}\varphi\rangle
\le \|L^{-1/2}F\|_{p} \|X_iL^{-1/2}\varphi\|_{q}
\\ &\le C \|L^{-1/2}F\|_{p} \|\varphi\|_{q}\le C \|F\|_{p} \|\varphi\|_{q},   \end{align*}
where we apply the $L^q$ boundedness of first order Riesz transforms in Theorem \ref{thm:LpRT} and Lemma \ref{lem:fracLaplacian-Lp} (ii). The proof is thus complete.

It remains to show that $X_iX_jU \in L^p(G)$, for $i,j\in I$. Note first that $X_iX_jU\in \mathcal T_L^{\,'}$.
Moreover, for any $\varphi\in \mathcal T_L \cap L^q(G)$, 
\[
\langle X_iX_jU, \varphi\rangle =\langle F, X_iX_jL^{-1}\varphi\rangle
\le \|F\|_{p} \|X_iX_jL^{-1}\varphi\|_{q}
\le C \|F\|_{p} \|\varphi\|_{q},
\]
where $q$ is the conjugate of $p$ and the last inequality follows from the $L^q$ boundedness of second order Riesz transforms in Theorem \ref{thm:Lp2ndRT}.  We thus obtain that $ X_iX_jU\in L^p(G)$.

Finally, the convergence follows from the same line as in the proof of Theorem \ref{thm:StrongSol-Poisson}.
\end{proof}

In view of the previous result, Theorem \ref{thm:Lip-Regularity-Lp} and Theorem \ref{thm:Lip-Regularity-Linfty}, we also conclude
\begin{prop}\label{prop:Lip-weak-glob}
Let $1\le p\le \infty$ and let $\theta>0$. Consider the Poisson equation \eqref{eq:Poisson-distribution}. 
  \begin{enumerate}
      \item If $1< p<\infty$, assume that $\Lambda_{\theta}^p(F)<\infty$ and $(\mu_t)_{t>0}$ satisfies the property $(\mathrm{CK})$. Then $\Lambda_{\theta+2}^p(U),\Lambda_{\theta+1}^p(X_iU), \Lambda_{\theta}^p(X_iX_jU)<\infty$ and   $-\sum_{i\in I}X_i^2 u$ converges to $Lu$ with respect to  $\Lambda_{\theta}^p(\cdot)$.
      \item If $p=1$ or $\infty$ and $0<\theta<1$, assume that $\Lambda_{\theta,2}^p(F)<\infty$ and $(\mu_t)_{t>0}$ satisfies $(\mathrm{CK}\lambda)$ such that $0<\lambda<\frac\theta2$. Then $\Lambda_{\theta+2,3}^{p}(U),\Lambda_{\theta+1-2\lambda,2}^{p}(X_iU),\Lambda_{\theta-2\lambda,1}^{p}(X_iX_jU)<\infty$ and $-\sum_{i\in I}X_i^2 u$ converges to $Lu$ with respect to  $\Lambda_{\theta-2\lambda,1}^{p}(\cdot)$.
  \end{enumerate}
\end{prop}

\subsection{Regularity for local distribution solutions}
Given an open set $\Omega$, let $\mathcal B_0(\Omega)$ be the set of all smooth cylindric functions with support in $\Omega$.
We consider the Poisson equation in the distribution sense, i.e., two distributions $U, F\in \mathcal T_L^{\,'}$ satisfy
\[
LU=F \quad \text{ in } \mathcal T_L^{\,'}.
\]

In this subsection, our main goal is to establish local  Lipschitz regularity results for the weak solution of the Poisson equation. As a non-technical  illustration, we obtain the following result (see Proposition \ref{prop:loc-weak-Lip} for a more concrete version). 
\begin{prop}
 Assume that $(\mu_t)_{t>0}$ is a symmetric central Gaussian semigroup satisfying condition $(\mathrm{CK0^+})$. Let $U,F\in \mathcal T_L^{\,'}$ satisfy the Poisson equation $LU=F$ in $\mathcal T_L^{\,'}$. Consider $\Omega\subset G$ such that $F\in L^{\infty}(\Omega)$ and $\mathrm L_{\theta}^{\infty}(F; \Omega)<\infty$  for $0<\theta\le1$. 
Let  $0<\beta<\theta$. Then for any open set $\Omega_1\Subset \Omega$ (i.e., $\overline{\Omega_1}\subset \Omega$), we have that for all $i,j\in I$ 
    \[
    \mathrm L_{\beta}^{\infty}(X_iU;\Omega_1), \,\mathrm L_{\beta}^{\infty}(X_iX_jU;\Omega_1)<\infty,
    \]
and $-\sum_{i\in I}X_i^2 U$ converges to $LU$ with respect to the seminorm $\mathrm L_{\beta}^{\infty}(\cdot;\Omega_1)$.
    
\end{prop}

We begin with the Sobolev regularity by assuming that $F\in L^p(\Omega)$, that is, there exists a fuction $f\in L^p(\Omega)$ such that $F(\phi)=\int_G f\phi d\nu$ for all $\phi\in \mathcal B_0(\Omega)$. 

\begin{prop}\label{prop:local-Lp}
Assume that $(\mu_t)_{t>0}$ is a symmetric central Gaussian semigroup satisfying condition $(\mathrm{CK*})$.
Let $U,F\in \mathcal T_L^{\,'}$ satisfy the Poisson equation $LU=F$ in $\mathcal T_L^{\,'}$. For $1<p<\infty$, let $ \Omega\subset G$ be such that $F\in L^p(\Omega)$. 
Then for any open set $\Omega_1\Subset \Omega$, we have 

\begin{enumerate}
\item $U\in L^p(\Omega_1)$;
\item $X_i U, X_iX_j  U \in L^p(\Omega_1)$, for any $i, j\in I$;
\item The series $-\sum_{i\in I}X_i^2 U$ converges to $LU$ in $L^p(\Omega_1)$.

\end{enumerate}
\end{prop}

\begin{proof}
We first introduce some notations. Let $\Omega_0$ be an open set such that $\Omega_0\Subset \Omega$. Fix a function $\eta_0\in \mathcal B_0(\Omega)$ such that $\eta_0\equiv 1$ on a neighborhood of $\Omega_0$, i.e., $\eta_0$ is a cutoff function of $\Omega_0$ in $\Omega$. Set
\begin{equation}\label{eq:def-tildeUFV}
\widetilde U=\eta_0 U, \quad, \widetilde F=\eta_0 F, \quad \widetilde V=L\widetilde U-\widetilde F.
\end{equation}
From this definition, it is enough to show (i), (ii) for  $\widetilde U$, $X_i\widetilde U, X_iX_j \widetilde U$. For the sake of clarity, we will distinguish the open set $\Omega_1$ in the proof of (i)--(iii) below.

\begin{enumerate}
\item  The proof is similar as in \cite[Section 3.1]{BSC2006} and we sketch the main steps here. 
First 
observe that $\widetilde U, \widetilde F, \widetilde V\in \mathcal T_L^{\,'}$. Moreover, $\widetilde V$ is supported in $\Omega\setminus \overline{\Omega_0}$.
We consider
\begin{equation}\label{eq:tilde-t}
\widetilde U^t=\widetilde U*\check{\mu}_t, \quad \widetilde F^t=\widetilde F*\check{\mu}_t, \quad \widetilde V^t=\widetilde V*\check{\mu}_t,
\end{equation}
where the convolution $*$ was defined in Section \ref{sec:distribution}.
It follows from  \cite[Lemma 3.3]{BSC2006} and Proposition \ref{prop:convolution} that  $\widetilde U^t,L\widetilde U^t, \widetilde F^t, \widetilde V^t$ are all in $\mathcal T_L$. Moreover
\begin{equation}\label{eq:deriv-U}
\partial_t \widetilde U^t=-L\widetilde U^t= -\widetilde F^t-\widetilde V^t.
\end{equation}

Fix $x_0\in \Omega_0$ and let $\Omega_1$ be a neighborhood of $x_0$ so $\Omega_1\Subset \Omega_0$. In order to prove $\widetilde U\in L^p(\Omega_1)$, it suffices to show that $\{\widetilde U^t\}_{t>0}$ is a Cauchy sequence in $L^p(\Omega_1)$ by bounding $\|\partial_t\widetilde U^t\|_{L^p(\Omega_1)}$. Indeed, we observe that
\[
\|\widetilde U\|_{L^p(\Omega_1)}\le \|\widetilde U^1\|_{L^p(\Omega_1)}+\sup_{0<t\le 1}\|\partial_t\widetilde U^t\|_{L^p(\Omega_1)}.
\]
Equivalently, by \eqref{eq:deriv-U} one needs to bound 
\[
\|\widetilde F^t\|_{L^p(\Omega_1)} \quad\text{and} \quad\|\widetilde V^t\|_{L^p(\Omega_1)}.
\]
The former is a consequence of Minkowski's inequality since 
\[
\|\widetilde F^t\|_{L^p(\Omega_1)} = \|\widetilde F*\check{\mu}_t\|_{L^p(\Omega_1)} \le \|\widetilde F\|_{L^p(\Omega_1)}<\infty.
\]
The latter can be bounded exactly the same way as \cite[Lemma 3.5]{BSC2006} by noting that 
\[
\|\widetilde V^t\|_{L^p(\Omega_1)}\le \|\widetilde V^t\|_{L^{\infty}(\Omega_1)}.
\]
For the sake of completeness, we write the detail here. Indeed, let $\Theta_0$ be an open neighborhood of $G\setminus \Omega_0$ such that $x_0\notin \overline{\Theta_0}$ and $e\notin \overline{\Theta_0^{-1}\Omega_1}$.
Since $ \widetilde V \in \mathcal T{\,'}(G)$ is supported in $\Omega\setminus\overline{\Omega_0}$, we can write
\begin{equation}\label{eq:Vx}
 \widetilde V^t(x)= \widetilde V *\check{\mu}_t(x)= \widetilde V(\mathcal L_{x^{-1}}\mu_t)=\widetilde V(\eta_1 \mathcal L_{x^{-1}}\mu_t)=\widetilde V(\check{\eta}_1 \mathcal R_{x}\check{\mu}_t),
\end{equation}
where $\eta_1\in \mathcal B_0(\Theta_0)$ and $\eta_1=1$ in a neighborhood of $\overline{\Omega\setminus \overline{\Omega_0}}$. Note that we use the notations $\mathcal L$ and $\mathcal R$ introduced in Section \ref{sec:distribution}. 

Recall the definition of $\widetilde V$ in \eqref{eq:def-tildeUFV}, i.e., $\widetilde V=L(\eta_0 U)+\eta_0 LU$. Then \cite[Lemma 3.2]{BSC2006} gives that 
\begin{equation}\label{eq:boundMk}
\widetilde V(\phi) \le C M_L^k(\phi), \quad \phi\in \mathcal T(G).
\end{equation}
Hence $\widetilde V(\check{\eta}_1 \mathcal R_{x}\check{\mu}_t) \le C M_L^k(\check{\eta}_1 \mathcal R_{x}\check{\mu}_t)$. Now if we assume $x\in \Omega_1$, then $y\in \Theta_0$ implies that $y^{-1}x\in \Theta_0^{-1}\Omega_1= \Theta$. 
In view of the definition in \eqref{eq:local-Mk} and the symmetry of $\mu_t$, we have 
\[
\widetilde V(\check{\eta}_1 \mathcal R_{x}\check{\mu}_t) \le C  M_L^k(\Theta, \mu_t), 
\]
which is finite under the condition $(\mathrm{CK}*)$, thanks to \cite[Corollary 4.11]{BSC2002}.

\item Next we turn to the proof for $X_i\widetilde U$. Denote $C=\int_G \widetilde F d\nu$, then one has $\widetilde F-C\in L_0^p(G)$. 
Letting 
\[
f=L^{-1}(\widetilde F-C),
\]
it follows from Proposition \ref{prop:distributionLp} that $f, X_i f, X_iX_j f\in L^p(G)$, for any $i, j\in I$.

On the other hand, we have
\[
L (\widetilde U-f)=\widetilde F+\widetilde V-\widetilde F+C=\widetilde V+C, 
\]
and therefore
\[
L(X_i\widetilde U-X_if)=X_i \widetilde V.
\]
Notice that $\widetilde V\equiv 0$ on $\Omega_0$. We deduce from part (i) that for any fixed open set $\Omega_2\Subset \Omega_0$ and $x_0\in \Omega_2$, $X_i\widetilde U-X_if\in L^p (\Omega_3)$,  for some neighborhood $\Omega_3\Subset \Omega_2$ of $x_0$.  Hence $X_i\widetilde U\in L^p(\Omega_3)$. Similarly, one can also obtain that $X_i X_j \widetilde U\in L^p(\Omega_3)$.

\item We take a cutoff function $\eta_2\in \mathcal B_0(\Omega_3)$ such that $\eta_2=1$ on $\Omega_4\Subset \Omega_3$. Denoting $\vardbtilde U=\eta_2 \widetilde U$ and $\vardbtilde F=\eta_2 \widetilde F$,  one has
\begin{equation}\label{eq:L-locU}
L\vardbtilde U=L(\eta_2 \widetilde U)=\eta_2 L \widetilde U+\widetilde U L\eta_2 -2\nabla \eta_2 \cdot \nabla \widetilde U.
\end{equation}
Recall that from the assumption and (i), the first two terms in the right hand side are in $L^p(G)$. It remains to show that $\nabla \eta_2 \cdot \nabla \widetilde U\in L^p(G)$. Indeed, since $\eta_2\in \mathcal B_0(\Omega_3)$, there exists $\alpha \in \aleph$ and $\varphi\in C^{\infty}(G_{\alpha})$ such that $\eta_2=\varphi\circ \pi_{\alpha}$. It follows from (ii) that 
\[
\|\nabla \eta_2 \cdot \nabla \widetilde U\|_{p}\le C_{\varphi}\|\Gamma_{\alpha}(\widetilde U,\widetilde U)^{1/2}\|_{L^p(\Omega_3)}<\infty.
\]

Hence the right hand side of \eqref{eq:L-locU} is in $L^p(G)$.
Now we can use the same strategy as in the proof of Proposition \ref{prop:distributionLp} (see also the proof of Theorem \ref{thm:StrongSol-Poisson}) to conclude that $-\sum_{i\in I}X_i^2 \vardbtilde U$ converges to $L\vardbtilde U$ in $L^p$. 
\end{enumerate}
\end{proof}
Next, we consider the regularity problem on the Lipschitz spaces $\mathrm L_{\theta}^{\infty}$. 

\begin{prop}\label{prop:Local-infinity}
Assume that $(\mu_t)_{t>0}$ is a symmetric central Gaussian semigroup satisfying condition $(\mathrm{CK}\lambda)$ for some $\lambda\in (0,1)$. Let $U,F\in \mathcal T_L^{\,'}$ satisfy the Poisson equation $LU=F$ in $\mathcal T_L^{\,'}$. Suppose that there exists $\Omega\in G$ such that $F\in L^{\infty}(\Omega)$, and for $0<\theta\le1$ there holds
\begin{equation}\label{eq:L-theta-infty}
\mathrm L_{\theta}^{\infty}(F; \Omega)\coloneqq\sup_{x,y\in \Omega}\frac{|F(x)-F(y)|}{d(x,y)^{\theta}}<\infty.   
\end{equation}
Then for any open set $\Omega_1\Subset \Omega$, we have
\[
\mathrm L_{\theta}^{\infty}( U; \Omega_1)\coloneqq\sup_{ x,y\in \Omega_1}\frac{|U(x)- U(y)|}{d(x,y)^{\theta}}<\infty.
\]
\end{prop}

\begin{remark}
    The assumption $(\mathrm{CK\lambda})$ can be replaced by the property $(\mathrm{CD})$, that is,  the intrinsic distance $d$ associated to $(\mu_t)_{t>0}$ (see Definition \ref{def:distance}) is continuous. In fact, property $(\mathrm{CD})$ implies  $(\mathrm{CK*})$, see for instance \cite{BSC2000, BSC2001} for more details.
\end{remark}

\begin{proof}
Let $\Omega_0, \widetilde U, \widetilde F, \widetilde V$ be the same as in the proof of Proposition \ref{prop:local-Lp}. Fix $x_0\in \Omega_0$ and let $\Omega_1\Subset \Omega_0$ be an open neighborhood of $x_0$.6
We note that $\mathrm L_{\theta}^{\infty}(\widetilde U; \Omega_1)=\mathrm L_{\theta}^{\infty}( U; \Omega_1)$. Hence it is enough to work for $\widetilde U$.
We use similar strategy as in the proof of Proposition \ref{prop:local-Lp} (i). Recalling the definitions of $\widetilde U^t$, $\widetilde F^t$ and $\widetilde V^t$ in \eqref{eq:tilde-t}, we aim to show that $\widetilde U^t$ is a Cauchy sequence with respect to the Lipschitz norm on $\Omega_1$. It suffices to bound the Lipschitz norms of $\widetilde F^t$ and $\widetilde V^t$ separately.

Since $\mu_t$ is central and symmetric, we also write $\widetilde F^t(x)=\int_{z\in G} \widetilde F(xz)d\mu_t(z)$. Then 
\begin{align}
\sup_{x,y\in \Omega_1}\frac{|\widetilde F^t(x)-\widetilde F^t(y)|}{d(x,y)^{\theta}}
& \le \int_G \sup_{x,y\in \Omega_1}\frac{|\widetilde F(xz)-\widetilde F(yz)|}{d(xz,yz)^{\theta}}d\mu_t(z) 
\\& \le C(\mathrm L_{\theta}^{\infty}(F; \Omega) +\|F\|_{L^{\infty}(\Omega)})
<\infty.
\end{align}
Indeed, the first inequality above is due to the bi-invariance of the distance so $d(x,y)=d(xz,yz)$, and the second one follows by observing that
\begin{equation}\label{eq:tildeF-Lip}
\frac{|\widetilde F(x)-\widetilde F(y)|}{d(x,y)^{\theta}}
\le |\eta_0(x)|\frac{| F(x)- F(y)|}{d(x,y)^{\theta}}+|F(y)|\frac{|\eta_0(x)-\eta_0(y)|}{d(x,y)^{\theta}}.
\end{equation}

Next we turn to the estimate of $ \widetilde V^t$. From \eqref{eq:Vx}, one has
$\widetilde V^t(x) =\widetilde V(\eta_1 \mathcal L_{x^{-1}}\mu_t)$,
where $\eta_1\in \mathcal B_0(\Theta_0)$ and $\eta_1=1$ in a neighborhood of $\overline{\Omega\setminus \overline{\Omega_0}}$ (recall that $\Theta_0$ is chosen such that $x_0\notin \overline{\Theta_0}$ and $e\notin \overline{\Theta_0^{-1}\Omega_1}$). Then \eqref{eq:boundMk} yields
\begin{align*}
|\widetilde V^t(x)-\widetilde V^t(y)|
&=|\widetilde V(\eta_1(\mathcal L_{x^{-1}}\mu_t -\mathcal L_{y^{-1}}\mu_t))| \le CM_L^k(\eta_1(\mathcal L_{x^{-1}}\mu_t -\mathcal L_{y^{-1}}\mu_t))
\\ &
=CM_L^k(\check{\eta}_1(\mathcal R_{x}\check{\mu}_t -\mathcal R_{y}\check{\mu}_t)).
\end{align*}

Since $\eta_1\in \mathcal B_0(\Theta_0)$, there exist $\beta\in \aleph$ and $\phi\in \mathcal C_0^{\infty}(G_{\beta})$ such that $\supp \phi\subset \pi_{\beta}(\Theta_0)$ and $\eta_1=\phi\circ \pi_{\beta}$.
Consider now $\mu_t^{\alpha}$, where $\alpha\ge \beta$. Let $x,y\in \Omega_1$. There exists a curve $\gamma_{xy}^{\alpha}(s): [0,T]\to G_{\alpha}$ connecting $\pi_{\alpha}(x)$ and $\pi_{\alpha}(y)$ such that $T=d_{\alpha}(\pi_{\alpha}(x),\pi_{\alpha}(y))$, $\gamma_{xy}^{\alpha}(0)=\pi_{\alpha}(x)$, $\gamma_{xy}^{\alpha}(T)=\pi_{\alpha}(y)$ and $\gamma_{xy}^{\alpha}(s)\in \pi_{\alpha}(\Omega_1)$. Moreover, for $\mu_t^{\alpha}$ one has
\[
\frac{d}{ds} \mu_t^{\alpha}(\gamma_{xy}^{\alpha}(s))=\sum_i a_iX_{\alpha,i}\mu_t^{\alpha}(\gamma_{xy}^{\alpha}(s)),
\]
where $\{a_i\}$ is a sequence of normalized coefficients so $\sum_i a_i^2\le 1$. We thus have for any $x\in \Omega_1$ and $z\in \Theta_0$ 
\begin{align*}
\mu_t^{\alpha}(z^{-1}x)-\mu_t^{\alpha}(z^{-1}y)
&=\mu_t^{\alpha}(\pi_{\alpha}(z^{-1})\pi_{\alpha}(x))-\mu_t^{\alpha}(\pi_{\alpha}(z^{-1})\pi_{\alpha}(y))
\\ &=-\int_0^T \frac{d}{ds}\mu_t^{\alpha}(\pi_{\alpha}(z^{-1})\gamma_{xy}^{\alpha}(s))ds 
\\ &=-\int_0^{T}\sum_ia_iX_{\alpha,i}\mu_t^{\alpha}(\pi_{\alpha}(z^{-1})\gamma_{xy}^{\alpha}(s))ds.
\end{align*}
It follows from Minkowski's inequality that 
\[
M_L^k\left((\check\phi\circ\pi_{\alpha})(\mathcal R_{\pi_{\alpha}(x)}\check\mu_t^{\alpha} -\mathcal R_{\pi_{\alpha}(y)}\check\mu_t^{\alpha})\right) 
\le \int_0^T M_L^k\left((\check\phi\circ\pi_{\alpha})\sum_i a_iX_{\alpha,i}\mathcal R_{\gamma_{xy}^{\alpha}(s)}\check\mu_t^{\alpha}\right) ds.
\]
Observing that $\gamma_{xy}^{\alpha}(s) \in \pi_{\alpha}(\Omega_1)$ and denoting $\Theta_{\alpha}:=\pi_{\alpha}(\Theta_0^{-1})\pi_{\alpha}(\Omega_1)$,  we have 
\[
M_L^k\left((\check\phi\circ\pi_{\alpha})\sum_i a_iX_{\alpha,i}\mathcal R_{\gamma_{xy}^{\alpha}(s)}\check\mu_t^{\alpha}\right)
\le M_L^k\left(\Theta_{\alpha},\sum_i a_i X_{\alpha,i} \mathcal R_{\gamma_{xy}^{\alpha}(s)}\check\mu_t^{\alpha}\right).
\]
Notice that the operator $P^{\ell, \lambda}$ defined in \eqref{eq:P-ell-lambda} is linear and recall the definition of $M_L^k(\Omega,f )$ in \eqref{eq:local-Mk}. It then follows from the Cauchy-Schwarz inequality that 
\begin{align*}
    M_L^k\left(\Theta_{\alpha},\sum_i a_i X_{\alpha,i} \mathcal R_{\gamma_{xy}^{\alpha}(s)}\check\mu_t^{\alpha}\right)
    &=
    \sup_{\substack{m,n\in \mathbb N\\n+2m\le k}} \sup_{\lambda\in \Lambda(n,m)}\sup_{x\in \Theta_{\alpha}} \left(\sum_{\ell\in I^n}\Big|P_{\alpha}^{\ell,\lambda}\sum_ia_iX_{\alpha,i}\mu_t^{\alpha}(x)\Big|^2\right)^{\frac12}
    \\&\le 
    \sup_{\substack{m,n\in \mathbb N\\n+2m\le k}} \sup_{\lambda\in \Lambda(n,m)} \sup_{x\in \Theta_{\alpha}} \left(\sum_{\ell\in I^{n}} \sum_i |P_{\alpha}^{\ell,\lambda}X_{\alpha,i}\mu_t^{\alpha}(x)|^2\right)^{\frac12}
    \\ &\le
    \sup_{\substack{m,n\in \mathbb N\\n+1+2m\le k+1}} \sup_{\lambda\in \Lambda(n+1,m)} \sup_{x\in \Theta_{\alpha}} \left(\sum_{\ell\in I^{n+1}}|P_{\alpha}^{\ell,\lambda}\mu_t^{\alpha}(x)|^2\right)^{\frac12}
    \\ &\le M_L^{k+1}\left(\Theta_{\alpha},\mu_t^{\alpha}\right),
\end{align*}
where $P_{\alpha}^{\ell,\lambda}$ is the projection of $P^{\ell,\lambda}$ onto $G_\alpha$.
Denote $\widetilde V_{\alpha}^t(x) =\widetilde V\left((\phi\circ \pi_{\alpha}) \mathcal L_{\pi_{\alpha}(x^{-1})}\mu_t^{\alpha}\right)$.
Collecting the above bounds, one obtains
\begin{align*}
|\widetilde V_{\alpha}^t(x)-\widetilde V_{\alpha}^t(y)|
    &\le C d_{\alpha}(\pi_{\alpha}(x),\pi_{\alpha}(y)) M_L^{k+1}\left(\Theta_{\alpha},\mu_t^{\alpha}\right).
\end{align*}

In order to conclude the proof, it remains to show the following claims:
\begin{enumerate}[(a)]
    \item $\widetilde V_{\alpha}^t(x)$ converges to $\widetilde V^t(x)$;
    \item For any $k\ge 0$, $M_L^{k}\left(\Theta_{\alpha},\mu_t^{\alpha}\right)$ converges to $M_L^{k}\left(\Theta,\mu_t\right)$, where $\Theta=\Theta_0^{-1}\Omega_1$.
\end{enumerate}
The first claim follows by recalling \eqref{eq:Vx} and noticing that $(\phi\circ \pi_{\alpha}) \mathcal L_{\pi_{\alpha}(x^{-1})}\mu_t^{\alpha}\in \mathcal B(G)$  converges to $\eta_1 \mathcal L_{x^{-1}}\mu_t$. For the latter, observe that as $e\notin \overline{\Theta}$, \cite[Corollary 4.11]{BSC2002} (see also Theorem \ref{thm:HKbounds}) implies that 
\[
M_L^{k}\left(\Theta,\mu_t\right)<\infty.
\]
Since $\mu_t^{\alpha}(x)=\mu_t^{\alpha}(\pi_{\alpha}(x))$ for any $x\in \Theta$, we have $M_L^{k}\left(\Theta_{\alpha},\mu_t^{\alpha}\right)=M_L^{k}\left(\Theta,\mu_t^{\alpha}\circ \pi_{\alpha}\right)$ which converges to $M_L^{k}\left(\Theta,\mu_t\right)$.

Thanks to the fact that $d(x,y)=\sup_{\alpha}d_{\alpha}(\pi_{\alpha}(x), \pi_{\alpha}(y))$, we then have
\begin{align*}
    \sup_{x,y\in \Omega_1} \frac{|\widetilde V^t(x)-\widetilde V^t(y)|}{d(x,y)^{\theta}} 
    &\le C \sup_{x,y\in \Omega_1}d(x,y)^{1-\theta} M_L^{k+1}\left(\Theta,\mu_t\right)<\infty
\end{align*}
provided that $0<\theta\le 1$ and the proof is completed. 
\end{proof}

\begin{prop}\label{prop:loc-weak-Lip}
Assume that $(\mu_t)_{t>0}$ is a symmetric central Gaussian semigroup satisfying condition $(\mathrm{CK\lambda})$ for some $\lambda\in (0,1)$. Let $U,F\in \mathcal T_L^{\,'}$ satisfy the Poisson equation $LU=F$ in $\mathcal T_L^{\,'}$. Consider $\Omega\in G$ such that  $F\in L^{\infty}(\Omega)$ and $\mathrm L_{\theta}^{\infty}(F; \Omega)<\infty$  for $0<\theta\le1$. Then for any open set $\Omega_1\Subset \Omega_0$,
\begin{enumerate}
    \item If $0<\lambda<\frac{\theta}{6+\theta}$, we have for all $i,j\in I$ 
    \[
    \mathrm L_{\frac{\beta-4\lambda}{1+3\lambda}}^{\infty}(X_iU;\Omega_1), \,\mathrm L_{\frac{\beta-4\lambda}{1+3\lambda}}^{\infty}(X_iX_jU;\Omega_1)<\infty,
    \]
    where $4\lambda<\beta<(1-\lambda)\theta-2\lambda$.
    \item If $0<\lambda<\frac{\theta}{24+\theta}$, we have $-\sum_{i\in I}X_i^2 U$ converges to $LU$ with respect to the seminorm $\mathrm L_{\frac{\beta'-4\lambda}{1+3\lambda}}^{\infty}(\cdot;\Omega_1)$, where  $4\lambda<\beta'<\frac{(1-\lambda)^2\theta-8\lambda}{1+3\lambda}$.
\end{enumerate}
\end{prop}
\begin{proof}
Similarly as in the proof of Proposition \ref{prop:local-Lp}, we will distinguish the open set $\Omega_1$ in the argument for (i) and (ii) below.
\begin{enumerate}
    \item 
 We use similar idea as in the proof of Proposition \ref{prop:local-Lp} (ii) and adopt the same notation there. Note that $\mathrm L_{\theta}^{\infty}(\widetilde F)<\infty$ (we need $\supp \eta_0\Subset \Omega$)  and consider $f=L^{-1}\widetilde F$. Then the following diagram indicates that $\mathrm L_{\frac{\beta-4\lambda}{1+3\lambda}}^{\infty}(X_if)<\infty$, where $0<\beta\le (1-\lambda)\theta-2\lambda$ is as in Theorem \ref{thm:LipSG<LipDist} (i).
\[
\begin{tikzcd}
 \mathrm L_{\theta}^{\infty}(\widetilde F)<\infty  
  \arrow[r,"\text{Thm. \ref{thm:LipSG<LipDist} (i)}","\text{Lem. \ref{lem:Lip-sg-p-alter-1infty} (i)}"'] & \Lambda_{\beta-2\lambda, 2}^{\infty}(\widetilde F)<\infty
 \arrow[d,"\text{Prop. \ref{prop:Lip-weak-glob} (ii)}"] \\
\Lambda_{\beta-4\lambda}^{\infty}(X_if) \arrow[d,"\text{Thm. \ref{thm:LipDist<LipSG-L1}}"] & \Lambda_{\beta+1-4\lambda,2}^{\infty}(X_if)<\infty \arrow[l,"\text{Lem. \ref{lem:Lip-sg-p-alter-1infty} (ii)}","\text{Rem. \ref{rem:Lip-sg-p-large} (iii)}"'] \\
\mathrm L_{\frac{\beta-4\lambda}{1+3\lambda}}^{\infty}(X_if)<\infty \arrow[r] & \mathrm L_{\frac{\beta-4\lambda}{1+3\lambda}}^{\infty}(X_if, \Omega_2)<\infty     \end{tikzcd}
\]
Next we observe that 
\begin{equation}\label{eq:XiUV}
L(X_i\widetilde U-X_if)=X_i \widetilde V.
\end{equation}
Since $X_i\widetilde V\equiv 0$ on $\Omega_1$, one has $\mathrm L_{\theta}^{\infty}(X_i\widetilde V,\Omega_1)<\infty$. In view of Proposition \ref{prop:Local-infinity}, it follows that there exists an open set $\Omega_2\Subset \Omega_1$ such that  $\mathrm L_{\theta}^{\infty}(X_i\widetilde U-X_if,\Omega_2)<\infty$. We conclude that $\mathrm L_{\frac{\beta-4\lambda}{1+3\lambda}}^{\infty}(X_i\widetilde U,\Omega_2)<\infty$ by noticing $\frac{\beta-4\lambda}{1+3\lambda}<\theta$.
The same proof with slight changes also deduces that $\mathrm L_{\frac{\beta-4\lambda}{1+3\lambda}}^{\infty}(X_iX_j\widetilde U,\Omega_2)<\infty$ for any $i,j\in I$.

\item Now we  show the convergence following the idea in  the proof of Proposition \ref{prop:local-Lp} (iii). Recall the equation \eqref{eq:L-locU}:
\[
L\vardbtilde U=L(\eta_2 \widetilde U)=\eta_2 L \widetilde U+\widetilde U L\eta_2 -2\nabla \eta_2 \cdot \nabla \widetilde U\coloneqq g,
\]
where we take a cutoff function $\eta_2\in \mathcal B_0(\Omega_3)$ such that $\eta_2=1$ on $\Omega_4\Subset \Omega_3$ for  $\Omega_3\Subset \Omega_2$, and $\vardbtilde U=\eta_2 \widetilde U$.

We first check that the three terms of $g$ are bounded with respect to $\mathrm L_{\frac{\beta-4\lambda}{1+3\lambda}}^{\infty}(\cdot)$. The first term is simple as we observe that $\eta_2 L \widetilde U=\eta_2\widetilde F=\eta_2 F$. Similarly as the estimate in \eqref{eq:tildeF-Lip}, one has
\begin{align*}
\sup_{x,y\in G}\frac{|(\eta_2F)(x)-(\eta_2\widetilde F)(y)|}{d(x,y)^{\theta}}
&\le \sup_{x,y\in G}|\eta_2(x)|\frac{| F(x)- F(y)|}{d(x,y)^{\theta}}+\sup_{x,y\in G}|F(y)|\frac{|\eta_2(x)-\eta_2(y)|}{d(x,y)^{\theta}}
\\&\le \mathrm L_{\theta}^{\infty}(F,\Omega)+\sup_{x\in \Omega_3,y\in G\setminus \Omega}\frac{| F(x)- F(y)|}{d(x,y)^{\theta}}+C\|F\|_{L^{\infty}(\Omega)}
\\&\le \mathrm L_{\theta}^{\infty}(F,\Omega)+C\|F\|_{L^{\infty}(\Omega)}.    
\end{align*}
Since we assume  $F\in L^{\infty}(\Omega)$ and $\mathrm L_{\theta}^{\infty}(F; \Omega)<\infty$, it is immediate to conclude that $\mathrm L_{\theta}^{\infty}(\eta_2 F)<\infty$ and hence $\mathrm L_{\frac{\beta-4\lambda}{1+3\lambda}}^{\infty}(\eta_2 L \widetilde U)<\infty$ from Remark \ref{rem:Lip-dis}.

For the second term, on the one hand the proof of Proposition \ref{prop:Local-infinity} gives that $\mathrm L_{\theta}^{\infty}(\widetilde U, \Omega_1)<\infty$. We also noticed that  $\widetilde U\in L^{\infty}(\Omega_1)$ comes from  the hypoellipticity of $L$ established in \cite[Section 3]{BSC2006}. Then the same proof as above gives that $\mathrm L_{\frac{\beta-4\lambda}{1+3\lambda}}^{\infty}(\widetilde U  L\eta_2)<\infty$.

Now we move to the last terms in $g$. Since $\eta_2\in \mathcal B_0(\Omega_3)$, there exists $\alpha \in \aleph$ and $\varphi\in C^{\infty}(G_{\alpha})$ such that $\eta_2=\varphi\circ \pi_{\alpha}$ and $\nabla \eta_2 \cdot \nabla \widetilde U=\sum_{i\in I_{\alpha}}X_i\eta_2X_i\widetilde U$. It is enough to prove that $\mathrm L_{\frac{\beta-4\lambda}{1+3\lambda}}^{\infty}(X_i\eta_2X_i\widetilde U)<\infty$ for each $i\in I_{\alpha}$. We have already shown in (i) that $\mathrm L_{\frac{\beta-4\lambda}{1+3\lambda}}^{\infty}(X_i\widetilde U,\Omega_2)<\infty$ for each $i\in I_{\alpha}$. In order to apply the same argument as above, one needs to verify that $X_i\widetilde U\in L^{\infty}(\Omega_2)$. Indeed, 
in view of \eqref{eq:XiUV} and the hypoellipticity of $L$ (\cite[Section 3]{BSC2006}), we get $X_i\widetilde U-X_if\in L^{\infty}(\Omega_2)$. Moreover, since $f=L^{-1}\widetilde F$, the hypoellipticity also yields $X_if\in L^{\infty}(\Omega_2)$. We conclude $X_i\widetilde U\in L^{\infty}(\Omega_2)$ and thus the proof.


Finally, one can repeat the first part of the proof and apply Proposition \ref{prop:Lip-weak-glob} to obtain the convergence result. Namely, denote $\theta'=\frac{\beta-4\lambda}{1+3\lambda}$, then  $-\sum_{i\in I}X_i^2 \vardbtilde U$ converges to $L\vardbtilde U$ with respect to the seminorm $\mathrm L_{\frac{\beta'-4\lambda}{1+3\lambda}}^{\infty}(\cdot)$ for $4\lambda<\beta'<(1-\lambda)\theta'-2\lambda$ and the assumption that $\lambda<\frac{\theta}{24+\theta}$ guarantees the existence of positive $\beta'$. The proof is completed by noticing that $\vardbtilde U=U$ on $\Omega_4$.
\end{enumerate}

\end{proof}

\bibliographystyle{plain}

\noindent \textbf{Alexander Bendikov}: \url{bendikov@math.uni.wroc.pl}\\
Institut of mathematics, Wroc\l{}aw University, Poland
\smallskip

\noindent \textbf{Li Chen}: \url{lchen@math.au.dk}\\
Department of Mathematics,
Aarhus University, Denmark
\smallskip

\noindent \textbf{Laurent Saloff-Coste}: \url{lsc@math.cornell.edu}\\
Department of mathematics
Cornell University, USA
\end{document}